\documentclass[12pt]{amsart} 

\usepackage{amsmath} \usepackage{amssymb} \usepackage{amsthm}
\usepackage{amscd} 
\usepackage[all]{xy} 
\usepackage{enumerate}
\usepackage{mathrsfs}
\usepackage{multirow}
\usepackage{bigdelim}
\usepackage[bookmarks=false]{hyperref}

\def\tb{\textbf}

\def\mc{\mathcal}
\def\ms{\mathscr}
\def\mb{\mathbb}
\def\mf{\mathfrak}

\def\a{{\alpha}}
\def\b{{\beta}}
\def\B{{\mathcal{B}}} 
\def\e{\varepsilon}
\def\E{{\mathcal{E}}} 
\def\F{{\mathcal{F}}} 
\def\G{{\mathcal{G}}}

\def\L{{\mathcal{L}}} 
 
\def\o{{\omega}}
\def\O{{\mathcal{O}}}
\def\T{{\mathcal{T}}}
\def\Z{{\mathbb{Z}}}

\def\Q{{\mathbb{Q}}} 
\def\R{{\mathbb{R}}} 

\def\vp{\varphi}
\def\ol{\overline}
\def\Tr{{\mathrm{Tr}}} 
 
\def\Pic{{\mathrm{Pic}}}

\def\im{{\mathrm{im}}} 
 
\def\codim{{\mathrm{codim}}} 
\def\coker{{\mathrm{coker}}}

\def\Proj{{\mathrm{Proj\; }}}

\theoremstyle{plain}
\newtheorem{thm}{Theorem}[section] 
\newtheorem{cor}[thm]{Corollary}
\newtheorem{prop}[thm]{Proposition}

\newtheorem{lem}[thm]{Lemma}
\theoremstyle{definition} 
\newtheorem{defn}[thm]{Definition}

\newtheorem{eg}[thm]{Example} 
\theoremstyle{remark}
\newtheorem{rem}[thm]{Remark}
\newtheorem{obs}[thm]{Observation}

\newtheorem{notation}[thm]{Notation}

\newtheorem*{acknowledgement}{Acknowledgments}

\title{Weak positivity theorem and Frobenius stable canonical rings of geometric generic fibers}
\author{Sho Ejiri}
\address{Graduate School of Mathematical Sciences, the University of Tokyo, 3-8-1 Komaba, Meguro-ku, Tokyo 153-8914, Japan.} 
\email{ejiri@ms.u-tokyo.ac.jp}

\begin{document}
\tolerance = 9999

\maketitle
\markboth{SHO EJIRI}{WEAK POSITIVITY THEOREM} 
\begin{abstract}
In this paper, we prove the weak positivity theorem in positive characteristic when the canonical ring of the geometric generic fiber $F$ is finitely generated and the Frobenius stable canonical ring of $F$ is large enough. As its application, we show the subadditivity of Kodaira dimensions in some new cases. 
\end{abstract}
\setcounter{tocdepth}{1}
\tableofcontents
\section{Introduction} \label{section:intro}
Let $f:X\to Y$ be a separable surjective morphism between smooth projective varieties over an algebraically closed field satisfying $f_*\O_X\cong\O_Y$. The positivity of the direct image sheaf $f_*\o_{X/Y}^m$ of the relative pluricanonical bundle is an important property. In characteristic zero, there are numerous known results. Fujita proved that $f_*\o_{X/Y}$ is a nef vector bundle when $\dim Y=1$ \cite{Fuj78}. Kawamata generalized this to the case when $m\ge2$ \cite{Kaw82} and to the case when $\dim Y\ge2$ \cite{Kaw81} (see also \cite{Fuj14}). Viehweg showed that $f_*\o_{X/Y}^m$ is weakly positive for each $m\ge1$ \cite{Vie83} (see also \cite{Kol87}, \cite{Cam04}, and \cite{Fuj13}). Here weak positivity is a property of coherent sheaves, which can be viewed as a generalization of nefness of vector bundles, and is equivalent to nefness when we restrict ourselves to vector bundles on smooth projective curves. 
There are several significant consequences of these results. One of them is Iitaka's conjecture in some special cases. Iitaka's conjecture states that the subadditivity of Kodaira dimensions $$\kappa(X)\ge\kappa(Y)+\kappa(X_{\ol\eta})$$ holds, where $X_{\ol\eta}$ is the geometric generic fiber of $f$ (note that this conjecture follows from a conjecture in the minimal model program \cite{Kaw85}). 
Other consequences include some moduli problems in \cite{Kol90} and \cite{Fuj12} (see also \cite{Vie95}), where results of \cite{Fuj78}, \cite{Kaw81}, and \cite{Kaw82} are generalized to the case when $X$ is reducible (see also \cite{Kaw11}, \cite{FF14}, and \cite{FFS14}). \par %
%
On the other hand, in positive characteristic, it is known that there are counter-examples to the above results.
 For example, Moret-Bailly constructed a semi-stable fibration $g:S\to \mathbb P^1$ from a surface $S$ to $\mathbb P^1$ such that $g_*\o_{S/{\mathbb P^1}}$ is not nef \cite{MB81}.  For other examples, see \cite{Ray78,Xie10} (or Remark \ref{rem:Raynaud-Xie} in this paper).  Hence it is natural to ask under what additional conditions analogous results hold in positive characteristic.  Koll\'ar showed that $f_*\o_{X/Y}^m$ is a nef vector bundle for each $m\ge2$ when $X$ is a surface, $Y$ is a curve, and the general fiber of $f$ has only nodes as singularities \cite[4.3.~Theorem]{Kol90}.  Patakfalvi proved that $f_*\o_{X/Y}^m$ is a nef vector bundle for each $m\gg0$ when $Y$ is a curve, $X_{\ol\eta}$ has only normal $F$-pure singularities, and $\o_{X/Y}$ is $f$-ample \cite[Theorem~1.1]{Pat14}. \par
In this paper, we consider the weak positivity of $f_*\o_{X/Y}^m$ in positive characteristic under a condition on the canonical ring and the Frobenius stable canonical ring of the geometric generic fiber. Recall that for a Gorenstein variety $V$, the canonical ring of $V$ is the section ring of the dualizing sheaf of $V$, and the Frobenius stable canonical ring of $V$ is its homogeneous ideal whose degree $m$ subgroup is $S^0(V,mK_V)$. $S^0(V,mK_V)$ is the subspace of $H^0(V,mK_V)$ defined by using the trace map of the Frobenius morphism (see Definition \ref{defn:S0} or \cite[\S 4]{Sch14}). These notions are naturally extended to pairs $(V,\Delta)$ consisted of Gorenstein varieties $V$ and effective $\Z_{(p)}$-Cartier divisors $\Delta$ on $V$. We use this setting throughout this paper. \par
From now on we work over an algebraically closed field of characteristic $p>0$. The following theorem is a main result of this paper.
\begin{thm}[\textup{Theorem \ref{thm:main thm}}]\label{thm:main thm intro}\samepage Let $f:X\to Y$ be a separable surjective morphism between smooth projective varieties, let $\Delta$ be an effective $\Q$-divisor on $X$ such that $a\Delta$ is integral for some integer $a>0$ not divisible by $p$, and let $\ol \eta$ be the geometric generic point of $Y$. Assume that \begin{itemize} \item[(i)] the $k(\ol\eta)$-algebra $\bigoplus_{m\ge0}H^0(X_{\ol\eta},m(aK_{X_{\ol\eta}}+(a\Delta)_{\ol\eta}))$ is finitely generated, and \item[(ii)] there exists an integer $m_0>0$ such that for each $m\ge m_0$, $$S^0(X_{\ol\eta},\Delta_{\ol \eta},m(aK_{X_{\ol \eta}}+(a\Delta)_{\ol \eta}))=H^0(X_{\ol\eta},m(aK_{X_{\ol \eta}}+(a\Delta)_{\ol \eta})).$$ \end{itemize} Then $f_*\O_X(am(K_{X/Y}+\Delta))$ is weakly positive for each $m\ge m_0$. \end{thm}
Condition (ii) holds, for example, in the case where $X_{\ol\eta}$ is a curve of arithmetic genus at least two which has only nodes as singularities, $\Delta=0$, and $m_0=2$ (Corollary \ref{cor:rs for F-pure curve}), or in the case where the pair $(X_{\ol\eta},\Delta_{\ol\eta})$ has only $F$-pure singularities, $K_{X_{\ol\eta}}+\Delta_{\ol\eta}$ is ample, and $m_0\gg0$ (Example \ref{eg:rs for ample div}). Thus Theorem \ref{thm:main thm intro} is a generalization of \cite[4.3.~Theorem]{Kol90} and \cite[Theorem~1.1]{Pat14}. \par 
Theorem \ref{thm:main thm intro} should be compared with another result of Patakfalvi \cite[Theorem~6.4]{Pat13}, which states that if $S^0(X_{\ol\eta},K_{X_{\ol\eta}})=H^0(X_{\ol\eta},K_{X_{\ol\eta}})$ then $f_*\o_{X/Y}$ is weakly positive (see also \cite{Jan08}). These two results imply that $S^0(X_{\ol\eta},mK_{X_{\ol\eta}})$ is closely related to the positivity of $f_*\o_{X/Y}^m$ for each $m\ge1$. In order to prove Theorem \ref{thm:main thm intro}, we generalize the method of the proof of \cite[Theorem~6.4]{Pat13} using a numerical invariant introduced in Section \ref{section:inv}. \par
When the relative dimension of $f$ is one, we obtain the following theorem as a corollary of Theorem \ref{thm:main thm intro}. 
\begin{thm}[\textup{Corollary \ref{cor:1 dim}}]\label{thm:1 dim intro}\samepage Let $f:X\to Y$ be a separable surjective morphism of relative dimension one between smooth projective varieties satisfying $f_*\O_X\cong \O_Y$, let $\Delta$ be an effective $\Q$-divisor on $X$ such that $a\Delta$ is integral for some integer $a>0$ not divisible by $p$, and let $\ol\eta$ be the geometric generic point of $Y$. If $(X_{\ol\eta},\Delta_{\ol\eta})$ is $F$-pure and $K_{X_{\ol\eta}}+\Delta_{\ol\eta}$ is ample, then $f_*\O_X(am(K_{X/Y}+\Delta))$ is weakly positive for each $m\ge2$. In particular, if $X_{\ol\eta}$ is smooth curve of genus at least two, then $f_*\o_{X/Y}^m$ is weakly positive for each $m\ge2$. \end{thm}
Moreover, in the case where $f:X\to Y$ is a semi-stable fibration from a surface to a curve, we discuss the ampleness of $f_*\o_{X/Y}^m$ for each $m\ge2$ and the nefness of $f_*\o_{X/Y}$ (Theorems \ref{thm:ample} and \ref{thm:criterion}). \par 
When the relative dimension of $f$ is two, we also obtain the following theorem as a corollary of Theorem \ref{thm:main thm intro}. 
\begin{thm}[\textup{Corollary \ref{cor:2 dim}}]\label{thm:2 dim intro}\samepage
Let $f:X\to Y$ be a separable surjective morphism of relative dimension two between smooth projective varieties satisfying $f_*\O_X\cong\O_Y$. If the geometric generic fiber is a smooth surface of general type and $p\ge 7$, then $f_*\o_{X/Y}^m$ is weakly positive for each $m\gg0$. \end{thm}
Similarly to the case of characteristic zero, we can use Theorem \ref{thm:main thm intro} to study Iitaka's conjecture. Before stating the next theorem, we recall the definition of Iitaka-Kodaira dimension. Let $D$ be a Cartier divisor on a projective variety $V$ and let $m>0$ be an integer divisible enough. The Iitaka-Kodaira dimension $\kappa(V,D)$ of $D$ is the dimension of the image of $V$ under the rational map determined by a linear system $|mD|$ if it is not empty, otherwise $\kappa(V,D)=-\infty$. This definition is naturally generalized to the case where $D$ is a $\Z_{(p)}$-Cartier (or $\Q$-Cartier) divisor. When $V$ is smooth, the Kodaira dimension $\kappa(V)$ of $V$ is defined as $\kappa(V,K_V)$. \par
\begin{thm}[\textup{Theorems \ref{thm:iitaka conj Y gen} and \ref{thm:iitaka conj Y ell}}]\label{thm:iitaka conj intro}\samepage Let $f:X\to Y$ be a separable surjective morphism between smooth projective varieties satisfying $f_*\O_X\cong\O_Y$, let $\Delta$ be an effective $\Q$-divisor on $X$ such that $a\Delta$ is integral for some integer $a>0$ not divisible by $p$, and let $\ol\eta$ be the geometric generic point of $Y$. Assume that \begin{itemize} \item[(i)] the $k(\ol\eta)$-algebra $\bigoplus_{m\ge0}H^0(X_{\ol\eta},m(aK_{X_{\ol\eta}}+(a\Delta)_{\ol\eta}))$ is finitely generated, \item[(ii)] there exists an integer $m_0\ge0$ such that $$S^0(X_{\ol\eta},\Delta_{\ol\eta},m(aK_{X_{\ol\eta}}+(a\Delta)_{\ol\eta}))=H^0(X_{\ol\eta},m(aK_{X_{\ol\eta}}+(a\Delta)_{\ol\eta}))$$ for each $m\ge m_0$, and \item[(iii)] either that $Y$ is of general type or $Y$ is an elliptic curve. \end{itemize} Then $$\kappa(X,K_X+\Delta)\ge\kappa(Y)+\kappa(X_{\ol\eta},K_{X_{\ol\eta}}+\Delta_{\ol\eta}).$$ \end{thm}
As a special case of Theorem \ref{thm:iitaka conj intro}, we obtain the following result.
\begin{thm}[\textup{Corollary \ref{cor:iitaka conj 3 1}}]\label{thm:iitaka conj 3 1}\samepage Let $f:X\to Y$ be a separable surjective morphism from a smooth projective variety $X$ of dimension three to a smooth projective curve $Y$ satisfying $f_*\O_X\cong\O_Y$. If the geometric generic fiber $X_{\ol\eta}$ is a smooth projective surface of general type and $p\ge7$, then $$\kappa(X)\ge\kappa(Y)+\kappa(X_{\ol\eta}).$$ \end{thm}
\subsection{Notation} \label{subsection:notat}
In this paper, we fix an algebraically closed field $k$ of characteristic $p>0$. A {\it $k$-scheme} is a separated scheme of finite type over $k$. A {\it variety} means an integral $k$-scheme and a {\it curve} (resp. {\it surface}) means a variety of dimension one (resp. two). A projective surjective morphism $f:X\to Y$ between varieties is called a {\it fibration} or a {\it algebraic fiber space} if it is separable and it satisfies $f_*\O_X\cong\O_Y$. \par
We fix the following notation: \begin{itemize} \item Let ${\rm CDiv}(S)$ be the group of the Cartier divisors on a scheme $S$. An $\Z_{(p)}$-Cartier (resp. $\Q$-Cartier) divisor on $S$ is an element of ${\rm CDiv}(S)\otimes_{\Z}\Z_{(p)}$ (resp. ${\rm CDiv}(S)\otimes_\Z\Q$), where $\Z_{(p)}$ is the localization of $\Z$ at the prime ideal $(p)=p\Z$; 
\item For a rational number $\delta$, we denote its integral (resp. fractional) part by $\lfloor\delta\rfloor$ (resp. $\{\delta\}$). For a $\Q$-divisor $\Delta=\sum_{i} \delta_i\Delta_i$ on a normal variety, we define $\lfloor\Delta\rfloor:=\sum_i \lfloor\delta_i\rfloor\Delta_i$ (resp. $\{\Delta\}:=\sum_i\{\delta_i\}\Delta_i$); 
\item Let $\varphi:S\to T$ be a morphism of schemes and let $T'$ be a $T$-scheme. Then we denote the second projection $S_{T'}:=S\times_{T}T'\to T'$ by $\varphi_{T'}$. For a Cartier (or $\Z_{(p)}$-Cartier, $\Q$-Cartier) divisor $D$ on $S$, the pullback of $D$ to $S_{T'}$ is denoted by $D_{T'}$ if it is well-defined. Similarly, for an $\O_S$-module homomorphism $\a:\F\to\G$, the pullback of $\alpha$ to $S_{T'}$ is denoted by $\a_{T'}:\F_{T'}\to \G_{T'}$; 
\item For a scheme $X$ of positive characteristic, $F_X:X\to X$ is the absolute Frobenius morphism. We often denote the source of $F_X^e$ by $X^{e}$. Let $f:X\to Y$ be a morphism between schemes of positive characteristic. We denote the same morphism by $f^{(e)}:X^e\to Y^e$ when we regard $X$ (resp. $Y$) as $X^e$ (resp. $Y^e$). We define the $e$-th relative Frobenius morphism of $f$ to be the morphism $F^{(e)}_{X/Y}:=(F_X^e,f^{(e)}):X^e\to X\times_Y Y^e=:X_{Y^e}$.  \end{itemize}
\begin{small}
\begin{acknowledgement}
The author wishes to express his gratitude to his supervisor Professor Shunsuke Takagi for suggesting the problems in this paper, for answering many questions, and for much helpful advice. 
He is grateful to Professors Yifei Chen, Yoshinori Gongyo and Zsolt Patakfalvi for valuable comments and discussions. 
He would like to thank Professor Akiyoshi Sannai and Doctors Takeru Fukuoka, Kenta Sato, and Fumiaki Suzuki for useful comments. 
He also wishes to thank the referee for careful reading and valuable suggestions.
He was supported by JSPS KAKENHI Grant Number 15J09117 and the Program for Leading Graduate Schools, MEXT, Japan.
\end{acknowledgement}
\end{small}
\section{Preliminaries} \label{section:prelim}
\subsection{AC divisors}\label{subsection:ACdiv}
Let $X$ be a $k$-scheme of pure dimension satisfying $S_2$ and $G_1$. An \textit{AC divisor} (or \textit{almost Cartier divisor}) on $X$ is a coherent $\O_X$-submodule of the sheaf of total quotient ring $K(X)$ which is invertible in codimension one (see \cite{Kol+92}, \cite{Har94}, or \cite{MS12}). For any AC divisor $D$ we denote the coherent sheaf defining $D$ by $\O_X(D)$. 
The set of AC divisors $\textrm{WSh}(X)$ has the structure of additive group \cite[Corollary~2.6]{Har94}. 
A $\Z_{(p)}$-\textit{AC divisor} is an element of $\textrm{WSh}(X)\otimes_\Z\Z_{(p)}$. 
An AC divisor $D$ is said to be \textit{effective} if $\O_X\subseteq\O_X(D)$, and a $\Z_{(p)}$-AC divisor $\Delta$ is said to be \textit{effective} if $\Delta=D\otimes r$ for some effective AC divisor $D$ and some $0\le r\in\Z_{(p)}$. 
Now we have the following diagram:
$$\xymatrix@R=25pt@C=25pt{ \textrm{WSh}(X) \ar[r]^(0.4){(\underline{~~})\otimes1} & \textrm{WSh}(X)\otimes_{\Z}\Z_{(p)} \\ \textrm{CDiv}(X) \ar[r]^(0.4){(\underline{~~})\otimes1} \ar@{^{(}-{>}}[u] & \textrm{CDiv}(X)\otimes_{\Z}\Z_{(p)} \ar@{^{(}-{>}}[u]}$$ 
Note that the horizontal homomorphisms are not necessarily injective \cite[Page 172]{Kol+92}.
\textbf{Throughout this paper, given an effective $\mathbb Z_{(p)}$-AC (resp. $\mathbb Z_{(p)}$-Cartier) divisor $\Delta$, we fix an effective AC (resp. Cartier) divisor $E$ and an integer $a>0$ not divisible by $p$ such that $E\otimes1=a\Delta$}. 
The choice of $E$ and $a$ is often represented by $\Delta=E/a$.
For every integer $m$, we regard the $\mathbb Z_{(p)}$-AC divisor $am\Delta$ as the AC divisor $mE$. 
For instance, the symbol $\O_X(am(D+\Delta))$ denotes the sheaf $\O_X(amD+mE)$, for every AC divisor $D$.\par
We note that if $X$ is a normal variety, then AC divisors are Weil divisors, and the horizontal homomorphisms in the above diagram is injective. In this case we can choose $E$ and $a$ canonically for an effective $\Z_{(p)}$-divisors $\Delta$: $a$ is the smallest positive integer such that $a\Delta$ is integral and $E:=a\Delta$. \par
We define notions similar to the above by using $\Q$ instead of $\Z_{(p)}$.
\subsection{Trace of Frobenius morphisms}\label{subsection:trace}
In this subsection we introduce notations related to trace maps of Frobenius morphisms. \par
Let $\pi:X\to Y$ be a finite surjective morphism between Gorenstein $k$-schemes of pure dimension, and let $\o_X$ and $\o_Y$ be dualizing sheaves of $X$ and $Y$ respectively. We denote by $\Tr_{\pi}:\pi_*\o_X\to\o_Y$ the morphism obtained by applying the functor $\ms Hom_{\O_Y}(\underline{\quad},\o_Y)$ to the natural morphism $\pi^{\#}:\O_Y\to\pi_*\O_X$. This is called the trace map of $\pi$.  Then we define 
\begin{align*} 
\phi^{(1)}_{X}&:=\Tr_{F_X}\otimes\O_X(-K_X):{F_X}_*\O_X((1-p)K_X)\to\O_X,\quad\textup{and} \\ 
\phi^{(e+1)}_{X}&:=\phi^{(e)}_{X}\circ {F_X^{e}}_*(\phi^{(1)}_{X}\otimes\O_X((1-p^{e})K_X)):{F_X^{e+1}}_*\O_X((1-p^{e+1})K_X)\to\O_X 
\end{align*} 
for each $e>0$, where $K_X$ is a Cartier divisor satisfying $\O_X(K_X)\cong\o_X$. \par
Let $X$ be a Gorenstein $k$-scheme of pure dimension. Let $\Delta=E/a$ be an effective $\Z_{(p)}$-Cartier divisor on $X$ and $d>0$ be the smallest integer satisfying $a|(p^d-1)$.  For each $e>0$ we define 
\begin{align*} 
\L^{(de)}_{(X,\Delta)}&:=\O_X((1-p^{de})(K_X+\Delta))\subseteq\O_X((1-p^{de})K_X),\\ 
\phi^{(d)}_{(X,\Delta)}&:{F_X^{d}}_*\L^{(d)}_{(X,\Delta)}\to {F_X^{d}}_*\O_X((1-p^{d})K_X)\xrightarrow{\phi^{(d)}_{X}}\O_X,\quad \textup{and} \\
\phi^{(d(e+1))}_{(X,\Delta)}&:=\phi^{(de)}_{(X,\Delta)}\circ{F_X^{de}}_*(\phi^{(d)}_{(X,\Delta)}\otimes\L^{(de)}_{(X,\Delta)}):{F_X^{d(e+1)}}_*\L^{(d(e+1))}_{(X,\Delta)}\to\O_X. 
\end{align*}
\indent Let $X$ be a $k$-scheme of pure dimension satisfying $S_2$ and $G_1$. Let $\Delta=E/a$ be a $\Z_{(p)}$-AC divisor on $X$ and $d>0$ be the smallest integer satisfying $a|(p^d-1)$. Let $\iota:U\hookrightarrow X$ be a Gorenstein open subset of $X$ such that $\codim X\setminus U\ge2$ and that $E|_U$ is Cartier. Set $\Delta|_U=E|_U/a$. Then for each $e>0$ we define 
\begin{align*} 
\L^{(de)}_{(X,\Delta)}:=\iota_*\L^{(de)}_{(U,\Delta|_U)}\quad\textup{and}\quad \phi^{(de)}_{(X,\Delta)}:=\iota_*(\phi^{(de)}_{(U,\Delta|_U)}):{F_X^{de}}_*\L^{(de)}_{(X,\Delta)}\to\O_X 
\end{align*} 
Note that $\phi^{(de)}_{(X,\Delta)}$ is a morphism between reflexive sheaves on $X$ (cf. \cite[Proposition~1.11]{Har94}). \par
\begin{defn}[\textup{ \cite[2.1. Definition]{Smi00}, \cite[Definition 3.1]{SS10} or \cite[Definition 2.6]{MS12}}]\label{defn:F-pure} 
With the notation as above, the pair $(X,\Delta)$ is said to be {\it sharply $F$-pure} (resp. {\it globally $F$-split}) if $\phi^{(e)}_{(X,\Delta)}$ is surjective (resp. split as $\O_X$-module homomorphism) for some $e>0$ satisfying $a|(p^e-1)$. 
We simply say that $X$ is $F$-pure (resp. globally $F$-split) if $(X,0)$ is $F$-pure (resp. globally $F$-split).  
\end{defn}
\begin{rem}\label{rem:F-pure} (1) $(X,\Delta)$ is $F$-pure if and only if $\phi^{(e)}_{(X,\Delta)}$ is surjective for any $e>0$ satisfying $a|(p^e-1)$. 
Indeed, if $\phi^{(e)}_{(X,\Delta)}$ is surjective for an $e>0$ with $a|(p^e-1)$, then $(\phi^{(e)}_{(X,\Delta)}\otimes\L^{(eg)}_{(X,\Delta)})^{**}$ is also surjective for any $g>0$, and so is $\phi^{(e(g+1))}_{(X,\Delta)}$. 
Here $(\underline{\quad})^{**}$ is the functor of the double dual. 
Let $e'>0$ be an integer with $a|(p^{e'}-1)$. 
Then $\phi^{(e')}_{(X,\Delta)}$ is surjective, since $\phi^{(ee')}_{(X,\Delta)}$ is surjective and factors through $\phi^{(e')}_{(X,\Delta)}$.\\
(2) The $F$-purity of $(X,\Delta)$ is independent of the choice of $E$ and $a$ satisfying $\Delta=E/a$. Let $E'$ be an effective AC divisor such that $E'=a'\Delta$ for an integer $a'>0$ not divisible by $p$. 
Let $g>0$ be an integer such that $q:=p^{eg}-1$ satisfies $a|q$, $a'|q$ and $qa^{-1}E=qa'^{-1}E'$. Then we see that 
$$\phi^{(eg)}_{(X,E/a)}\cong\phi^{(eg)}_{(X,qa^{-1}E/q)}\cong\phi^{(eg)}_{(X,qa'^{-1}E'/q)}\cong\phi^{(eg)}_{(X,E'/a')}.$$ 
Thus $\phi^{(e')}_{(X,E'/a')}$ is surjective.\\
(3) A statement similar to (1) and (2) holds for the global $F$-splitting of $(X,\Delta)$.
\end{rem}
\subsection{Trace of relative Frobenius morphisms}\label{subsection:trace_rel}
In this subsection, we introduce notations related to trace maps of relative Frobenius morphisms. See \cite{PSZ13} for more details on trace maps of relative Frobenius morphisms. \par
Let $f:X\to Y$ be a morphism between Gorenstein $k$-schemes of pure dimension. We assume that either $F_Y$ is flat (i.e., $Y$ is regular) or $f$ is flat. Then $F_Y$ or $f$ is a Gorenstein morphism, so $X_{Y^1}$ is a Gorenstein $k$-scheme \cite[III, \S9]{Har66}. We define the relative dualizing sheaf $\o_{X/Y}$ of $f$ to be $\o_X\otimes f^*\o_Y^{-1}$. 
Then we have \begin{align*}\o_{X_{Y^1}/Y^1}:=&\o_{X_{Y^1}}\otimes {f_{Y^1}}^*\o_{Y^1}^{-1}\\\cong&\o_{X_{Y^1}}\otimes {f_{Y^1}}^*\o_{Y^1}^{p-1}\otimes {f_{Y^1}}^*{F_Y}^*\o_Y^{-1} \\\cong& {(F_Y)_X}^!\O_X\otimes{(F_Y)_X}^*\o_X\otimes ({f_{Y^1}}^*{F_Y}^!\O_Y)^{-1} \otimes {(F_Y)_X}^*f^*\o_Y^{-1} \\\cong& {(F_Y)_X}^*\o_{X/Y}=(\o_{X/Y})_{Y^1} \end{align*} by the assumption. Moreover, for positive integers $d,e$, we consider the following commutative diagram: 
\newpage
$$ 
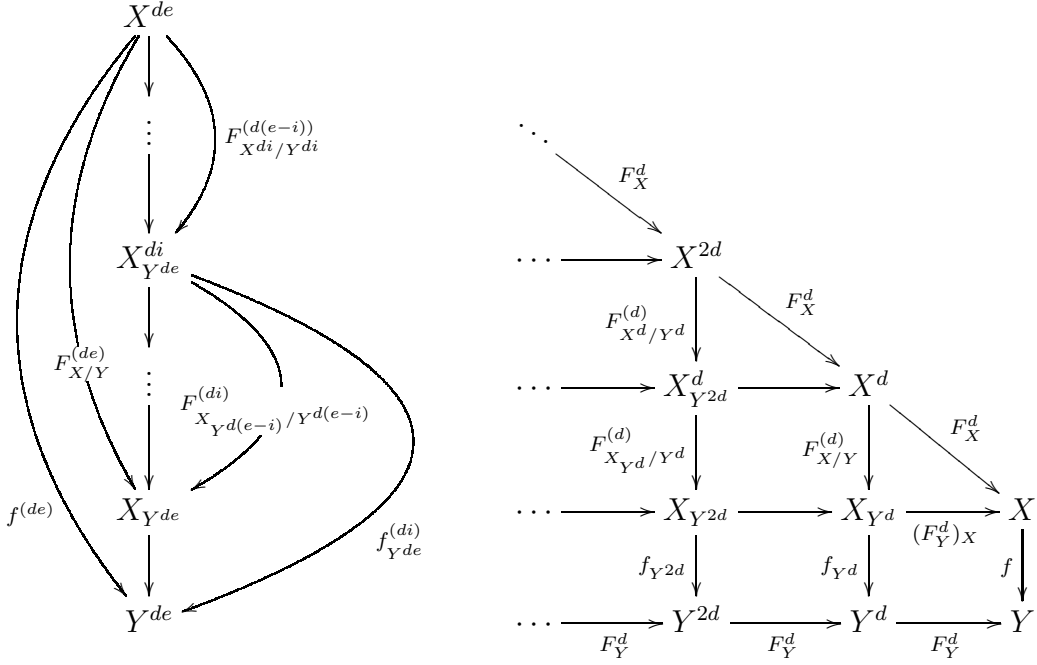
\xymatrix@R=25pt @C=35pt{ X^{de} \ar[d] \ar@/^25pt/[dd]^(0.5){F_{X^{di}/Y^{di}}^{(d(e-i))}} \ar@/_30pt/[dddd]|(0.7){F_{X/Y}^{(de)}} \ar@/_50pt/[ddddd]_(0.8){f^{(de)}} &&&&&&\\ \vdots \ar[d] &&& \ddots \ar[dr]^{F_X^d} & & & \\ X^{di}_{Y^{de}} \ar[d] \ar@/^50pt/[dd]|(0.6){F_{X_{Y^{d(e-i)}}/Y^{d(e-i)}}^{(di)}} \ar@/^100pt/[ddd]^(0.7){f_{Y^{de}}^{(di)}} && & \cdots \ar[r] & X^{2d} \ar[dr]^{F_X^d} \ar[d]_{F_{X^d/Y^d}^{(d)}} & & \\ \vdots \ar[d] && & \cdots \ar[r] & X^d_{Y^{2d}} \ar[r] \ar[d]_{F_{X_{Y^d}/Y^d}^{(d)}} & X^d \ar[dr]^{F_X^d} \ar[d]_{F_{X/Y}^{(d)}} \\ X_{Y^{de}} \ar[d] && & \cdots \ar[r] & X_{Y^{2d}} \ar[r] \ar[d]_{f_{Y^{2d}}} & X_{Y^d} \ar[r]_{(F_Y^d)_X} \ar[d]_{f_{Y^d}} & X \ar[d]_f \\ Y^{de} && & \cdots \ar[r]_{F_Y^d} & Y^{2d}\ar[r]_{F_Y^d} & Y^d \ar[r]_{F_Y^d} & Y \\ } $$
Let $K_{X/Y}$ be a Cartier divisor on $X$ satisfying $\O_X(K_{X/Y})\cong\o_{X/Y}$. This is denoted by $K_{X^e/Y^e}$ when we regard $f$ as $f^{(e)}$. Set $K_{X^e_{Y^g}/Y^g}:=(K_{X^e/Y^e})_{Y^g}$ for each $g\ge e$. Then for each $e>0$ we define 
\begin{align*} \phi^{(1)}_{X/Y}&:=\Tr_{F^{(1)}_{X/Y}}\otimes\O_{X_{Y^1}}(-K_{X_{Y^1}}):{F^{(1)}_{X/Y}}_*\O_{X^1}((1-p)K_{X^1/Y^1})\to\O_{X_{Y^1}},\quad\textup{and} \\ \phi^{(e+1)}_{X/Y}&:=\left(\phi^{(e)}_{X/Y}\right)_{Y^{e+1}}\circ {F^{(e)}_{X_{Y^1}/Y^1}}_*\left(\phi^{(1)}_{X^{e}/Y^{e}}\otimes\O_{X^e_{Y^{e+1}}}((1-p^{e})K_{X^{e}_{Y^{e+1}}/Y^{e+1}})\right)\\ & \qquad: {F^{(e+1)}_{X/Y}}_*\O_X((1-p^{e+1})K_{X^{e+1}/Y^{e+1}})\to\O_{X_{Y^{e+1}}}. 
\end{align*}
\indent 
Let $\Delta=E/a$ be an effective $\Z_{(p)}$-AC divisor on $X$ and $d$ be the smallest positive integer satisfying $a|(p^d-1)$. For each $e>0$ we define 
\begin{align*} 
\L^{(de)}_{(X,\Delta)/Y}&:=\O_{X^{de}}((1-p^{de})(K_{X^{de}/Y^{de}}+\Delta))\subseteq\O_{X^{de}}((1-p^{de})K_{X^{de}/Y^{de}}), \\ 
\phi^{(d)}_{(X,\Delta)/Y}&:{F^{(d)}_{X/Y}}_*\L^{(d)}_{(X,\Delta)/Y}\to{F^{(d)}_{X/Y}}_*\O_{X^d}((1-p^d)K_{X^{d}/Y^{d}})\xrightarrow{\phi^{(d)}_{X/Y}}\O_{X_{Y^d}},\quad\textup{and}\\ 
\phi^{(d(e+1))}_{(X,\Delta)/Y}&:=\left(\phi^{(de)}_{(X,\Delta)/Y}\right)_{Y^{d(e+1)}}\circ{F^{(de)}_{X_{Y^{d}}/Y^{d}}}_*\left(\phi^{(d)}_{(X^{de},\Delta)/Y^{de}}\otimes\left(\L^{(de)}_{(X,\Delta)/Y}\right)_{Y^{d(e+1)}}\right) \\ 
&:{F^{(d(e+1))}_{X/Y}}_*\L^{(d(e+1))}_{(X,\Delta)/Y}\to \O_{X_{Y^{d(e+1)}}}.
\end{align*} 
\indent
Let $f:X\to Y$ be a morphism between $k$-schemes of pure dimension. Assume that $X$ satisfies $S_2$ and $G_1$, $Y$ is Gorenstein, and $f$ or $F_Y$ is flat. 
Let $\Delta=E/a$ be an effective $\Z_{(p)}$-AC divisor on $X$ and $d$ be the smallest positive integer satisfying $a|(p^d-1)$. 
Let $\iota:U\hookrightarrow X$ be a Gorenstein open subset of $X$ such that $\codim X\setminus U\ge2$ and that $E|_U$ is Cartier. Set $\Delta|_U=E|_U/a$. Then for each $e>0$ we define 
\begin{align*}
\L^{(de)}_{(X,\Delta)/Y}&:={\iota_{Y^{de}}}_*\L^{(de)}_{(U,\Delta|_U)/Y},\quad\textup{and} \\
\phi^{(de)}_{(X,\Delta)/Y}&:={\iota_{Y^{de}}}_*(\phi^{(de)}_{(U,\Delta|_U)}/Y):{F^{(de)}_{X/Y}}_*\L^{(de)}_{(X,\Delta)/Y}\to\O_{X_{Y^{de}}}.
\end{align*} 
\section{Frobenius stable canonical ring} \label{section:rs}
In this section, we introduce and study Frobenius stable canonical rings. 
After definitions and basic properties, we study Frobenius stable canonical rings of varieties with ample canonical bundles. Especially, we consider the case of Gorenstein projective curves (Corollary \ref{cor:rs for F-pure curve}). 
We also discuss the case of varieties with semi-ample canonical bundles in any dimension (Corollary \ref{cor:Iitaka fibrat}). To this end, we prove Theorem \ref{thm:quasi can bdl formula} which is a kind of canonical bundle formula. 
As another application of the theorem, we study Frobenius stable canonical rings of surfaces of general type (Corollary \ref{cor:surf of gen type}).
\begin{notation}\label{notation:pair}
Let $X$ be a $k$-scheme of pure dimension satisfying $S_2$ and $G_1$, and let $\Delta=E/a$ be an effective $\Z_{(p)}$-AC divisor. Set $d>0$ be the smallest integer satisfying $a|(p^d-1)$. 
\end{notation}
\begin{defn}[\textup{\cite[\S 3]{Sch14}}]\label{defn:S0}
In the situation of Notation $\ref{notation:pair}$, let $\mc M$ be a reflexive sheaf on $X$ of rank one such that invertible in codimension one. Then we define $S^0(X,\Delta,\mc M)$ as 
\begin{align*} \bigcap_{e>0}\im\left(H^0(X,(({F^{de}_X}_*\L^{(de)}_{(X,\Delta)})\otimes\mc M)^{**})\xrightarrow{H^0(X,(\phi^{(de)}_{(X,\Delta)}\otimes\mc M)^{**}) } H^0(X,\mc M)\right), 
\end{align*} 
where $\phi^{(de)}_{(X,\Delta)}$ is the morphism defined in Subsection \ref{subsection:trace}, and $(\underline{\quad})^{**}:=\ms Hom(\ms Hom(\underline{\quad},\O_X),\O_X)$ is the functor of the double dual. For any AC divisor $D$ on $X$, we denote $S^0(X,\Delta,\O_X(D))$ by $S^0(X,\Delta,D)$. Write $S^0(X,D):=S^0(X,0,D)$. \end{defn}
\begin{rem}\label{rem:S0} 
The above definition does not depend on the choice of $E$ and $a$ satisfying $\Delta=E/a$. Indeed, if $E'$ and $a'$ satisfy $\Delta=E'/a'$, then by an argument similar to Remark \ref{rem:F-pure} (2) we have $\phi^{(eg)}_{(X,E/a)}\cong\phi^{(eg)}_{(X,E'/a')}$ for every $g>0$ divisible enough.
\end{rem}
\begin{eg}\normalfont\label{eg:gl F-sp} In the situation of Notation $\ref{notation:pair}$, it is easily seen that the following are equivalent: \begin{itemize}\item[(1)] $(X,\Delta)$ is globally $F$-split. \item[(2)] $S^0(X,\Delta,\O_X)=H^0(X,\O_X).$ \item[(3)] $S^0(X,\Delta,D)=H^0(X,D)$ for every AC divisor $D$ on $X$. \end{itemize}\end{eg}
\begin{defn}[\textup{\cite[Section~4.1]{HP13} or \cite[Exercise 4.13]{PST14}}] \label{defn:rs}
In the situation of Notation $\ref{notation:pair}$, let $\mc M$ be a reflexive sheaf on $X$ of rank one such that invertible in codimension one. Then we define \begin{align*} R_S(X,\Delta,\mc M):=&\bigoplus_{n\ge0}S^0(X,\Delta,(\mc M^{\otimes n})^{**}) \subseteq R(X,\mc M):=\bigoplus_{n\ge0}H^0(X,(\mc M^{\otimes n})^{**}).\end{align*} For any AC divisor $D$, we denote $R(X,\O_X(D))$ and $R_S(X,\Delta,\O_X(D))$ respectively by $R(X,D)$ and $R_S(X,\Delta,D)$. $R_S(X,\Delta,a(K_X+\Delta))$ is called the {\it Frobenius stable canonical ring}, where $K_X$ is an AC divisor such that $\O_X(K_X)$ is isomorphic to the dualizing sheaf $\o_X$ of $X$. \\
When $D$ is a $\Q$-Weil divisor on a normal variety $X$, we define 
$$R_S(X,\Delta,D):=\bigoplus_{n\ge0}S^0(X,\Delta,\lfloor nD\rfloor)\subseteq R(X,D):=\bigoplus_{n\ge0}H^0(X,\lfloor nD\rfloor).$$
\end{defn}
\begin{lem}[\textup{\cite[Lemma~4.1.1]{HP13}}] \label{lem:ideal} $R_S(X,\Delta,D)$ is an ideal of $R(X,D)$. \end{lem}
\begin{proof} This follows from an argument similar to the proof of \cite[Lemma~4.1.1]{HP13}. \end{proof}
\begin{notation}\label{notation:r/rs}
We denote by $R/R_S(X,\Delta,D)$ the quotient ring of $R(X,D)$ modulo $R_S(X,\Delta,D)$.
\end{notation}
We recall that the assumption (ii) of the main theorem (Theorem \ref{thm:main thm intro}): there exists an $m_0>0$ such that $S^0(X_{\ol\eta},\Delta_{\ol\eta},am(K_{X_{\ol\eta}}+\Delta_{\ol\eta}))=H^0(X_{\ol\eta},am(K_{X_{\ol\eta}}+\Delta_{\ol\eta}))$ for every $m\ge m_0$.
This is equivalent to the condition that there exists an integer $m_0>0$ such that the degree $m$ part of $R/R_S(X_{\ol\eta},\Delta_{\ol\eta},a(K_{X_{\ol\eta}}+\Delta_{\ol\eta}))$ is zero for every $m\ge m_0$.
Note that the existence of such $m_0$ is equivalent to the finiteness of the dimension of $k$-vector space $R/R_S(X_{\ol\eta},\Delta_{\ol\eta},a(K_{X_{\ol\eta}}+\Delta_{\ol\eta}))$
\begin{defn}\label{defn:fg}
In the situation of Notation $\ref{notation:pair}$, assume that each connected component of $X$ is integral. An AC divisor $D$ is said to be {\it finitely generated} if $R(X,D)$ is a finitely generated $k$-algebra. A $\Z_{(p)}$-AC (resp. $\Q$-AC) divisor $\Gamma$ is said to be {\it finitely generated} if there exists a finitely generated AC divisor $D$ such that $\Gamma=D\otimes\lambda$ for some $0<\lambda\in\Z_{(p)}$ (resp. $\Q$). \end{defn} 
\begin{lem}\label{lem:fg}
Let $R=\bigoplus_{m\ge0}R_m$ be a graded ring. Assume that $R$ is a domain and $R_0$ is a field. \begin{itemize} \item[(1)] If the $n$-th Veronese subring $R^{(n)}:=\bigoplus_{m\ge0}R_{mn}$ is a finitely generated $R_0$-algebra for some $n>0$, then so is $R$. \item[(2)] Let $\mf a\subseteq R$ be a nonzero homogeneous ideal, and suppose that $R$ is a finitely generated $R_0$-algebra. If $R^{(n)}/\mf a^{(n)}$ is a finite dimensional $R_0$-vector space for some $n>0$, then so is $R/\mf a$, where $\mf a^{(n)}:=\bigoplus_{m\ge0} \mf a_{mn}$. \end{itemize}\end{lem}
\begin{proof} For the proof of (1) we refer the proof of \cite[Lemma 5.68]{HK10}.
For (2), let $l>0$ be an integer divisible enough. Then there exists $n_0>0$ such that $\mf a_{l+n}\subseteq R_{l+n}=R_{n}\cdot R_{l}=R_n\cdot\mf a_{l}\subseteq \mf a_{l+n}$ for each $n\ge n_0$, and hence $\mf a_{m}=R_m$ for each $m\gg0$, which is our claim. \end{proof}
As mentioned after Notation \ref{notation:r/rs}, the assumption (ii) of the main theorem (Theorem \ref{thm:main thm intro}) satisfied if and only if $R/R_S(X_{\ol\eta},\Delta_{\ol\eta},a(K_{X_{\ol\eta}}+\Delta_{\ol\eta}))$ is finite dimensional as $k$-vector space. This condition is equivalent to the condition that $R/R_S(X_{\ol\eta},\Delta_{\ol\eta},an(K_{X_{\ol\eta}}+\Delta_{\ol\eta}))$ is finite dimensional for an integer $n>0$ by (2) of the above lemma.
\begin{defn}
In the situation of Notation $\ref{notation:pair}$, we denote the kernel of $\phi^{(de)}_{(X,\Delta)}:{F_X^{de}}_*\L^{(de)}_{(X,\Delta)}\to \O_X$ by $\B^{de}_{(X,\Delta)}$ for every integer $e>0$. When $\Delta=0$, we denote $\B^{e}_{(X,0)}$ by $\B^{e}_X$. \end{defn} 
\begin{eg}\normalfont\label{eg:rs for ample div} 
In the situation of Notation $\ref{notation:pair}$, assume that $X$ is projective, and $(p^{c}-1)(K_X+\Delta)$ is Cartier for some $c>0$ divisible by $d$. 
Let $H$ be an ample Cartier divisor. We show that $R/R_S(X,\Delta,H)$ is finite dimensional if and only if $(X,\Delta)$ is $F$-pure. 
By the Fujita vanishing theorem, there is an $m>0$ such that $H^1(X,\B^{c}_{(X,\Delta)}(mH+N))=0$ for every nef Cartier divisor $N$. 
We may assume that $mH-(K_X+\Delta)$ is nef. 
If $(X,\Delta)$ is $F$-pure, or equivalently if $\phi^{(c)}_{(X,\Delta)}$ is surjective, then so is the morphism $H^0(X,\phi^{(c)}_{(X,\Delta)}\otimes\O_X(mH+N))$. 
Furthermore we see that $H^0(X,\phi^{(ce)}_{(X,\Delta)}\otimes\O_X(mH+N))$ is also surjective for each $e>0$, because of the definition of $\phi^{(ce)}_{(X,\Delta)}$ and the following isomorphisms 
\begin{align*} 
&\left({F_X^{ce}}_*(\phi^{(c)}_{(X,\Delta)}\otimes\L^{(ce)}_{(X,\Delta)})\right)\otimes\O_X(mH+N)\\ \cong&{F_X^{ce}}_*\left(\phi^{(c)}_{(X,\Delta)}\otimes\O_X(mH+(p^{ce}-1)(mH-(K_X+\Delta))+p^{ce}N)\right). 
\end{align*} 
This implies that $S^0(X,\Delta,mH+N)=H^0(X,mH+N)$ and that $R/R_S(X,\Delta,H)$ is finite dimensional.
Conversely it is clear that if $R/R_S(X,\Delta,H)$ is finite dimensional, then $\phi^{(c)}_{(X,\Delta)}$ is surjective, or equivalently, $(X,\Delta)$ is $F$-pure. 
\end{eg}
The above example shows that if $(X_{\ol\eta},\Delta_{\ol\eta})$ is $F$-pure and $K_{X_{\ol\eta}}+\Delta_{\ol\eta}$ is an ample $\Z_{(p)}$-Cartier divisor, then the assumption (ii) (and (i)) of the main theorem (Theorem \ref{thm:main thm intro}) holds. 
We next consider the value of such $m_0$ in the case when $X_{\ol\eta}$ a curve. Corollary \ref{cor:rs for F-pure curve} provides a value of such $m_0$ effectively when $K_{X_{\ol\eta}}+\Delta_{\ol\eta}$ is ample.
\begin{lem}\label{lem:rs for ample div on curve} \samepage 
Let $X$ be a Gorenstein projective curve, and let $H$ be an ample Cartier divisor such that $H-K_X$ is nef. Then for each integer $e,m\ge1$, $$ H^1(X,\B^e_X\otimes\O_X(K_X+mH))=0. $$ Moreover if $X$ is $F$-pure, then $$S^0(X,K_X+mH)=H^0(X,K_X+mH).$$ 
\end{lem}
\begin{proof} 
Clearly the second statement follows from the first and the long exact sequence of cohomology induced from the surjective morphism $\phi^{(e)}_{(X,\Delta)}\otimes\O_X(K_X+mH)$. 
We prove the first statement. Let $\nu:C\to X$ be the normalization. Then a commutative diagram of varieties 
\[ \xymatrix{ C \ar[r]^{F_C^e} \ar[d]_\nu & C \ar[d]^\nu \\ X \ar[r]_{F_X^e} & X } \] induces a commutative diagram of $\O_X$-modules:
\[ \xymatrix{  0 \ar[r] & \nu_*\B^e_C(K_C) \ar[r] \ar[d]^\a & \nu_*{F_C^e}_*\o_C \ar[r]^{\nu_*\Tr_{F_C^e}} \ar[d]^{{F_X^e}_*\Tr_\nu} & \nu_*\o_C \ar[r] \ar[d]^{\Tr_{\nu}} & 0 \\  0 \ar[r] & \B^e_X(K_X) \ar[r] & {F_X^e}_*\o_X \ar[r]_{\Tr_{F_X^e}} & \o_X & }\]  
Since each vertical morphism is an isomorphism on some dense open subset of $X$, the kernel and the cokernel of $\a$ are torsion $\O_X$-modules. 
Furthermore since $\B^e_C$ has no torsion, we see that $\a$ is injective. For each $m>0$, the following exact sequence 
$$ 0 \to (\nu_*\B^e_C(K_C))(mH) \xrightarrow{\a\otimes\O_X(mH)} \B^e_X(K_X+mH) \to \coker(\a) \to 0, $$ induces a surjection 
$$ H^1(C,\B^e_C(K_C+m\nu^*H))\cong H^1(X,(\nu_*\B^e_C(K_C))(mH)) \twoheadrightarrow H^1(X,\B^e_X(K_X+mH)).$$ 
Moreover, since $\nu^*H$ is ample and $$\nu^* H-K_C=\nu^*(H-K_X)+\nu^*K_X-K_C\sim \nu^*(H-K_X)+E$$ is nef, where $E$ is effective divisor on $C$ defined by the conductor ideal, we may assume that $X$ is smooth. 
Then we have $ H^1(X,mpH)=H^0(X,K_X-mpH)=0 $ for each $m\ge1$ by the Serre duality. For each $m\ge1$ there exists an exact sequence 
$$ 0 \to \O_X(mH) \to {F_X}_*\O_X(mpH) \to \B^1_X(K_X+mH) \to 0 $$ induced by Cartier operator, which shows that $H^1(X,\B^1_X(K_X+mH))=0$. 
This implies $H^0(X,\phi^{(1)}_{X}\otimes\O_X(K_X))$ is surjective, and thus $H^0(X,\phi^{(e)}_X\otimes\O_X(K_X+mH))$ is also surjective for every $e,m\ge 1$ because of the definition of $\phi^{(e)}_{X}$. 
Hence the exact sequence $$ 0 \to \B^e_X(K_X+mH) \to {F_X^e}_*\O_X(K_X+mp^eH) \xrightarrow{\phi^{(e)}_X\otimes\O_X(K_X+mH)} \O_X(K_X+mH) \to 0 $$ 
induces the following:
\begin{align*} H^1(X,\B^e_X(K_X+mH)) \hookrightarrow H^1(X,\O_X(K_X+mp^eH)) \cong H^0(X,-mp^eH) = 0.  \end{align*} 
\end{proof} 
\begin{prop}\label{prop:rs for F-pure curve} \samepage 
In the situation of Notation $\ref{notation:pair}$, let $X$ be a projective curve, let $K_X+\Delta$ is nef and let $H$ be a Cartier divisor. Assume either that $(\rm i)$ $H+(a-1)K_X$ is ample and $H+(a-2)K_X$ is nef, or that $(\rm ii)$ $X\cong\mathbb P^1$ and $H$ is ample. Then for each $e>0$, $$H^1(X,\B^{de}_{(X,\Delta)}\otimes\O_X(a(K_X+\Delta)+H))=0. $$ Moreover if $(X,\Delta)$ is $F$-pure, then $$S^0(X,\Delta,a(K_X+\Delta)+H)=H^0(X,a(K_X+\Delta)+H).$$ \end{prop}
\begin{proof} 
Clearly the second statement follows from the first and the long exact sequence of cohomology induced from the surjective morphism $\phi^{(e)}_{(X,\Delta)}\otimes\O_X(K_X+mH)$. 
We prove the first statement. Let $E'$ be an effective Cartier divisor satisfying $\O_X(E')\subseteq\O_X(E)$ and $\Delta':=E'/a$.
For each $e>0$ there is a commutative diagram $$ \xymatrix@R=25pt @C=20pt{ {F_X^{de}}_*\L^{(de)}_{(X,\Delta)}(p^{de}(a(K_X+\Delta)+H)) \ar@/^4pt/[r]^(0.55){\phi^{(de)}_{(X,\Delta)}\otimes\O_X(a(K_X+\Delta)+H)} & \O_X(a(K_X+\Delta)+H) \\ {F_X^{de}}_*\L^{(de)}_{(X,\Delta)}(p^{de}(a(K_X+\Delta')+H)) \ar@/^4pt/[r]^(0.55){\phi^{(de)}_{(X,\Delta')}\otimes\O_X(a(K_X+\Delta')+H)} \ar@{^(->}[u] & \O_X(a(K_X+\Delta')+H) \ar@{^(->}[u] }$$  
where the vertical morphisms are natural inclusion. 
This induces the injective morphism $$ \B^{de}_{(X,\Delta')}(a(K_X+\Delta')+H)\to\B^{de}_{(X,\Delta)}(a(K_X+\Delta)+H)$$
whose cokernel is a torsion $\O_X$-module. Hence it suffices to prove that $H^1(X,\B^{de}_{(X,\Delta)}(a(K_X+\Delta')+H))=0$. When (i) holds, we set $E'=0$. By the previous lemma we have $H^1(X,\B^{de}_{X}(aK_X+H))=0$. When (ii) holds, we may assume $a(K_X+\Delta')\sim0$. Then it is easily seen that $\dim H^1(X,\B^{de}_{(X,\Delta')})\le 1$. Since every vector bundle on $\mathbb{P}^1$ is isomorphic to a direct sum of line bundles, we have $H^1(X,\B^{de}_{(X,\Delta')}(H))=0$. This completes the proof.  
\end{proof} 
The following corollary will be used to prove weak positivity theorem for fibrations of relative dimension one (Corollary \ref{cor:1 dim}).
\begin{cor}\label{cor:rs for F-pure curve} 
In the situation of Notation $\ref{notation:pair}$, assume that $X$ is a projective curve and $(X,\Delta)$ is $F$-pure. If $K_X+\Delta$ is ample $($resp. $K_X$ is ample and $a\ge2$$)$, then for each $m\ge 2$ $($resp. $m\ge 1$$)$, $$S^0(X,\Delta,am(K_X+\Delta))=H^0(X,am(K_X+\Delta)).$$ 
\end{cor}
\begin{proof} We note that a Gorenstein curve has nef dualizing sheaf unless it is isomorphic to $\mathbb P^1$. Hence the statement follows from the above proposition. \end{proof}
\begin{rem} \label{rem:rs for F-pure curve} 
$F$-pure singularities of curves are completely classified \cite{GW77}. For example, nodes are $F$-pure singularities, but cusps are not.
\end{rem}
We next study Frobenius stable canonical rings of varieties with semi-ample canonical bundles in any dimension.
For such varieties, Corollary \ref{cor:Iitaka fibrat} provides a criterion of the finiteness of the dimension of $R/R_S$ in terms of the singularity of the canonical models.
This is obtained as an application of Theorem \ref{thm:quasi can bdl formula}, which is a kind of canonical bundle formula. \\
In order to formulate the problem, we start with an observation of Iitaka fibrations.
\begin{obs}\label{obs:Iitaka fibrat}
Let $X$ be a normal projective variety, and let $\Delta$ be an effective $\Z_{(p)}$-Weil divisor on $X$ such that $K_X+\Delta$ is a semi-ample $\Q$-Cartier divisor.  Let $$f:X\to Y:=\Proj R(X,K_X+\Delta)$$ be the Iitaka fibration.  Then there exists an ample $\Q$-Cartier divisor $H$ on $Y$ satisfying $f^*H\sim K_X+\Delta$.  Let $Y_0\subseteq Y$ be an open subset such that $f_0:=f|_{X_0}:X_0\to Y_0$ is flat, where $X_0:=f^{-1}(Y_0)$. \\
(I) Assume that $R_S(X,\Delta,K_X+\Delta)\ne0$. Then there exists an integer $m>0$ such that $m\Delta$ is integral and $S^0(X,\Delta,m(K_X+\Delta))\ne0$.  This implies that $$S^0(X,\Delta,((m-1)p^{e'}+1)(K_X+\Delta))\ne0$$ for some $e'>0$ divisible enough. Since $p\nmid (m-1)p^{e'}+1$, there exists an $e>0$ such that $S^0(X,\O_X((p^e-1)(K_X+\Delta)))\ne0$.  
We set $R':=(1-p^e)(K_X+\Delta)$. Let $\eta$ be the generic point of $Y$.  
By the assumption, $\O_X(-R')|_{X_{\eta}}$ is a torsion line bundle on $X_{\eta}$ with nonzero global sections, and thus it is trivial.  
Hence $\O_X(R')|_{X_{\eta}}$ is also trivial, and $f_*\O_X(R')$ is a torsion free sheaf on $Y$ of rank one.  Then there exists an effective Weil divisor $B$ supported on $X\setminus X_0$ such that $f_*\O_X(R'+B)\cong\O_Y(S)$ for some Weil divisor $S$ on $Y$.  We set $R:=R'+B=(1-p^e)(K_X+\Delta)+B$.  Then \begin{align*} R=K_X+\Delta+B-p^e(K_X+\Delta)\sim_{\Q}K_X+\Delta-p^ef^*H. \end{align*} Replacing $e$, we may assume that $p^eH$ is $\Z_{(p)}$-Cartier, and thus there exists an integer $a>0$ not divisible by $p$ such that $a\Delta$ is integral and $H':=ap^eH$ is Cartier. Then we have $$ aR\sim a(K_X+\Delta)+aB-f^*H'.$$ 
\noindent (I\hspace{-1pt}I) In the situation of (I), after replacing $e$ by its multiple, we assume that $(p^e-1)(K_X+\Delta)$ is base point free. Then we may take $(p^e-1)H$ as Cartier. In this case we have $\O_X(R')\cong f^*\O_Y((1-p^e)H)$, and thus we may choose $B=0$, $R=R'$ and $S=(1-p^e)H$ by projection formula. In particular we have $R\sim f^*S$.
\end{obs}
In a more general situation than the above, we prove the following theorem which is a kind of canonical bundle formula (see \cite[Theorem B]{DS15} for a related result). 
\begin{thm}\label{thm:quasi can bdl formula}
Let $f:X\to Y$ be a fibration between normal varieties, let $\Delta$ be an effective $\Q$-Weil divisor on $X$ such that $a\Delta$ is integral for some integer $a>0$ not divisible by $p$, and let $Y_0$ be a smooth open subset of $Y$ such that $\codim Y\setminus Y_0\ge2$ and $f_0:=f|_{X_0}:X_0\to Y_0$ is flat, where $X_0:=f^{-1}(Y_0)$.  Further assume that the following conditions: 
\begin{itemize} 
\item[(i)]$(X_{\ol \eta},\Delta_{\ol \eta})$ is globally $F$-split where ${\ol \eta}$ is geometric generic point of $Y$. 
\item[(ii)] There exists a Weil divisor $R$ on $X$, such that $f_*\O_X(R)\cong\O_Y(S)$ for some Weil divisor $S$ on $Y$ and $aR\sim a(K_X+\Delta)+B-f^*C$ for some effective Weil divisor $B$ supported on $X\setminus X_0$ and for some Cartier divisor $C$ on $Y$. 
\end{itemize}
Then, there exists an effective $\Q$-Weil divisor $\Delta_Y$ on $Y$, which satisfies the following conditions: 
\begin{itemize} 
\item[(1)] $a'\Delta_Y$ is integral for some integer $a'>0$ divisible by $a$ but not by $p$, and $$ \O_Y(a'(K_Y+\Delta_Y-S))\cong \O_Y(a'a^{-1}C)\cong f_*\O_X(a'(K_X+\Delta+a^{-1}B-R)).$$ 
\item[(2)]For every effective Weil divisor $B'$ supported on $X\setminus X_0$ and for every Cartier divisor $D$ on $Y$, $$S^0(X,\Delta,B'+f^*D+R)\cong S^0(Y,\Delta_Y,D+S).$$ 
\item[(3)]If $f$ is a birational morphism, then $\Delta_Y=f_*\Delta$.  
\item[(4)]Suppose that $X_0$ is Gorenstein and $R|_{X_0}$ is Cartier. Let $\Gamma$ be an effective Cartier divisor on $X_0$ defined by the image of the natural morphism $$ \O_{X_0}(-R|_{X_0})\otimes{f_0}^*({f_0}_*\O_{X_0}(R|_{X_0}))\to \O_{X_0},$$ and let $y$ be a point of $Y_0$. Then the following conditions are equivalent: \begin{itemize} \item[(a)]${\rm Supp\;}\Delta$ does not contain any irreducible component of $f^{-1}(y)$, and $(X_{\ol y},\Delta_{\ol y})$ is globally $F$-split, where $\ol y$ is the algebraic closure of $y$.  \item[(b)]$y$ is not contained in $f({\rm Supp\;} \Gamma)\cup {\rm Supp\;} \Delta_Y$.  \end{itemize} \end{itemize} \end{thm}
Note that if $R$ is linearly equivalent to the pullback of a Cartier divisor on $Y$, then replacing $C$, we may assume that $R=0$, $S=0$ and $\Gamma=0$.\par
For varieties with semi-ample canonical bundles, Corollary \ref{cor:Iitaka fibrat} provides a criterion of the finiteness of the dimension of $R/R_S$ in terms of the singularity of the canonical models. As explained after Notation \ref{notation:r/rs}, the finiteness of $R/R_S$ is equivalent to the assumption (ii) of the main theorem (Theorems \ref{thm:main thm intro} or \ref{thm:main thm}). We remark that for such varieties, the assumption (i) of the main theorem, that is the finitely generation of canonical rings, is always satisfied.
\begin{cor}\label{cor:Iitaka fibrat} 
In the situation of Observation \ref{obs:Iitaka fibrat} $(\mathrm{I})$, 
\begin{itemize} \item[(1)] $(X_{\ol\eta},\Delta_{\ol\eta})$ is globally $F$-split, where $\ol\eta$ is the geometric generic point of $Y$. In particular, $f$ is separable.  
\item[(2)] Let $\Delta_Y$ be as in Theorem $\ref{thm:quasi can bdl formula}$. In the situation of Observation \ref{obs:Iitaka fibrat} $(\mathrm{I\hspace{-1pt}I})$ $(${\rm i.e.} $l(K_X+\Delta)$ is Cartier for an integer $l>0$ not divisible by $p$$)$, $R/R_S(X,\Delta,K_X+\Delta)$ is a finite dimensional $k$-vector space if and only if $(Y,\Delta_Y)$ is $F$-pure.
\end{itemize} \end{cor}
Before the proof of Theorem \ref{thm:quasi can bdl formula} and Corollary \ref{cor:Iitaka fibrat}, we observe morphisms induced by the push-forward of the trace map of the relative Frobenius morphism. This observation will be also referred in the proof of Theorem \ref{thm:criterion}.
\begin{obs}\label{obs:lower semi conti}
Let $f:X\to Y$ be a projective morphism from a Gorenstein variety $X$ to a smooth variety $Y$. Let $\Delta=E/a$ be an effective $\Z_{(p)}$-AC divisor on $X$ whose support does not contain any irreducible component of any fiber of $f$. Let $d$ be the smallest positive integer satisfying $a|(p^d-1)$. Let $e\ge0$ be an integer.\\
(I) For every $y\in Y$, we have the following diagram: $$ \xymatrix { (X_{\ol y})^{de}& X^{de}_{\ol y^{de}} \ar[r] \ar[d]_(0.5){F_{X_{\ol y}/{\ol y}}^{(de)}} \ar@{=}[l]& X^{de} \ar[d]_(0.5){F_{X/Y}^{(de)}} \ar@/^8mm/[dd]^(0.3){f^{(de)}}& \\ & X_{\ol y^{de}} \ar[r] \ar[d] & X_{Y^{de}} \ar[d]_{f_{Y^{de}}} & \\ & \ol y^{de} \ar[r] & Y^{de} & \\ }$$ Let $R$ be a Cartier divisor on $X$. We denote by $\theta^{(de)}$ the morphism $${f_{Y^{de}}}_*(\phi^{(de)}_{(X,\Delta)/Y}\otimes\O_{X}(R)_{Y^{de}}):{f^{(de)}}_*\L^{(de)}_{(X,\Delta)/Y}(p^{de}R)\to {f_{Y^{de}}}_*\O_{X_{Y^{de}}}(R)_{Y^{de}}.$$ Here we recall that $\L^{(de)}_{(X,\Delta)/Y}:=\O_{X^{de}}((1-p^{de})(K_{X^{de}/Y^{de}}+\Delta))$. \\
(I\hspace{-1pt}I) Let $Y_0\subseteq Y$ be an open subset such that $f_0:=f|_{X_0}:X_0\to Y_0$ is flat, where $X_0:=f^{-1}(Y_0)$. 
Assume that $y\in Y_0$ and that $E|_{Y_0}$ is Cartier. Since $f_0$ is a Gorenstein morphism, $X_{\ol y}$ is Gorenstein.
Set $\Delta_{\ol y}=E|_{X_{\ol y}}/a$.  Then $\L^{(de)}_{(X,\Delta)/Y}|_{(X_{\ol y})^{de}}\cong\L^{(de)}_{(X_{\ol y}/\ol y,\Delta_{\ol y})}$ and we have the following diagram of $k(\ol y^{de})$-vector spaces for every $e>0$: $$ \xymatrix@R=60pt@C=30pt{ H^0((X_{\ol y})^{de},\L^{(de)}_{(X_{\ol y}/\ol y,\Delta_{\ol y})}\otimes \O_X(p^{de}R)|_{(X_{\ol y})^{de}}) \ar[d]|{H^0(X_{\ol y^{de}},\phi^{(de)}_{(X_{\ol y}/{\ol y},\Delta_{\ol y})}\otimes\O_X(p^{de}R)|_{X_{\ol y^{de}}})} & \left({f^{(de)}}_* \L^{(de)}_{(X,\Delta)/Y}(p^{de}R) \right)\otimes k({\ol y^{de}}) \ar[l] \ar[d]|(0.55){{\theta^{(de)}}\otimes k(\ol y^{de})}  \\ H^0(X_{\ol y^{de}},\O_X(R)|_{X_{\ol y^{de}}}) & \left( {f_{Y^{de}}}_*\O_X(R)_{Y^{de}}\right)\otimes k(\ol y^{de}) \ar[l] }$$ 
(I\hspace{-1pt}I\hspace{-1pt}I) Let $Y_1\subseteq Y_0$ be an open subset such that $\dim H^0(X_y,\O_X(R)|_{X_{y}})$ is a constant function on $Y_1$ with value $h$.  
If $y\in Y_1$, then the horizontal morphisms in the above diagram are isomorphisms by \cite[Corllary12.9]{Har77}. Hence for every $e>0$ we have \begin{align*} &\dim_{k(\ol y)}\im(H^0(X_{\ol y},\phi^{(de)}_{(X_{\ol y},\Delta_{\ol y})}\otimes\O_X(R)|_{X_{\ol y}})) \\ =&\dim_{k(\ol y^{de})}\im(H^0(X_{\ol y^{de}},\phi^{(de)}_{(X_{\ol y}/\ol y,\Delta_{\ol y})}\otimes\O_X(R)|_{X_{\ol y^{de}}})) \\ =& \dim_{k({\ol y^{de}})}\im({\theta^{(de)}}\otimes k({\ol y^{de}}))\\ =& h-\dim_{k(\ol y^{de})}\coker\left(\theta^{(de)}\otimes k(\ol y^{de})\right) \\ =& h-\dim_{k(\ol y^{de})}\left(\coker(\theta^{(de)})\right)\otimes k(\ol y^{de}).  \end{align*} Here, the last equality follows from the right exactness of the tensor functor. \\
(I\hspace{-1pt}V) Assume that $(p^d-1)(K_{X/Y}+\Delta-R)|_{X_1}\sim f_1^*C$ for some Cartier divisor $C$ on $Y_1$, where $X_1:=f^{-1}(Y_1)$ and $f_1:=f|_{X_1}:X_1\to Y_1$. Then $$\L^{(de)}_{(X_{\ol y},\Delta_{\ol y})}\otimes\O_{X^{de}}(p^{de}R)|_{(X_{\ol y})^{de}}\cong\O_{X^{de}}(R)|_{(X_{\ol y})^{de}}$$ for every $y\in Y_1$.  Thus we can regard $H^0(X_{\ol y},\phi^{(de)}_{(X_{\ol y},\Delta_{\ol y})}\otimes\O_X(R)|_{X_{\ol y}})$ as the $e$-th iteration of the ($p^{-d}$-linear) morphism $$\tau:=H^0(X_{\ol y},\phi^{(d)}_{(X_{\ol y},\Delta_{\ol y})}\otimes\O_X(R)|_{X_{\ol y}}):H^0(X_{\ol y},\O_X(R)|_{X_{\ol y}})\to H^0(X_{\ol y},\O_X(R)|_{X_{\ol y}}).$$ If $e\ge h$, then $\im(\tau^e)=\im(\tau^{h})$, and thus  $$\im(H^0(X_{\ol y},\phi^{(de)}_{(X_{\ol y},\Delta_{\ol y})}\otimes\O_X(R)|_{X_{\ol y}}))=S^0(X_{\ol y},\Delta_{\ol y},\O_X(R)|_{X_{\ol y}}).$$ Hence by (3), we see that $$ \dim_{k(\ol y)} S^0(X_{\ol y},\Delta_{\ol y},\O_X(R)|_{X_{\ol y}})=h-\dim_{k(\ol y^{de})}\left(\coker(\theta^{(de)})\right)\otimes k(\ol y^{de}). $$ In particular, since the function $\dim_{k(\ol y^{de})}(\coker(\theta^{(de)}))\otimes k(y^{de})$ on $Y^{de}$ is upper semicontinuous, the function $\dim_{k(\ol y)} S^0(X_{\ol y},\Delta_{\ol y},\O_X(R)|_{X_{\ol y}})$ on $Y_1$ is lower semicontinuous. 
\end{obs}
\begin{proof}[Proof of Theorem $\ref{thm:quasi can bdl formula}$] Let $d>0$ be an integer such that $a|(p^d-1)$. \\
\tb{Step1.} We define $\Delta_Y$ and we show that this is independent of the choice of $d$. We first note that, for each $e\ge0$ there exist isomorphisms \begin{align*} & f_*\O_{X}((1-p^{de})(K_{X}+\Delta)+p^{de}R)\\ \cong & f_*\O_{X}((1-p^{de})(K_{X}+\Delta+a^{-1}B-R)+(p^{de}-1)a^{-1}B+R)\\ \cong &\O_{Y}((1-p^{de})a^{-1}C)\otimes f_*\O_{X}((p^{de}-1)a^{-1}B+R) \\\cong & \O_{Y}((1-p^{de})a^{-1}C)\otimes f_*\O_{X}(R) \\ \cong & \O_{Y}((1-p^{de})a^{-1}C+S).\end{align*} 
Since $Y$ is normal, to define $\Delta_Y$ we may assume $Y=Y_0$ and $X$ is smooth. Then for each $e>0$ we have \begin{align*}f^{(de)}_*\L^{(de)}_{(X,\Delta)/Y}(p^{de}R)&\cong\O_{Y^{de}}((1-p^{de})(a^{-1}C-K_{Y^{de}})+S)\quad\textup{and}\\{f_{Y^{de}}}_*\O_{X_Y^{de}}(R_{Y^{de}})&\cong{F_Y^{de}}^*f_*\O_X(R)\cong\O_{Y^{de}}(p^{de}S), \end{align*} thus \begin{align*} \theta^{(de)}:={f_{Y^{de}}}_*\phi^{(de)}_{(X,\Delta)/Y}\otimes\O_{X_{Y^{de}}}(R_{Y^{de}}):f^{(de)}_*\L^{(de)}_{(X,\Delta)/Y}(p^{de} R)\to {f_{Y^{de}}}_*\O_{X_{Y^{de}}}(R_{Y^{de}}) \end{align*} is a homomorphism between line bundles. %
By the assumption of the global $F$-splitting of $(X_{\ol\eta},\Delta_{\ol\eta})$, we see that the left vertical morphism of the diagram in Observation \ref{obs:lower semi conti} (I\hspace{-1pt}I) (for $y=\eta$) is surjective, and hence by Observation \ref{obs:lower semi conti} (I\hspace{-1pt}I\hspace{-1pt}I) $\theta^{(de)}$ is generically surjective for every $e>0$. 
Thus $\theta^{(de)}$ defines an effective Cartier divisor $E^{(de)}$ on $Y$. Then for every $e>1$ we have $E^{(de)}=p^dE^{(d(e-1))}+E^{(d)}$, because relations between morphisms%
\begin{align*} \theta^{(de)}:=&{f_{Y^{de}}}_*\phi^{(de)}_{(X,\Delta)/Y}\otimes\O_{X_{Y^{de}}}(R_{Y^{de}})\\ =&{f_{Y^{de}}}_*\left(\phi^{(d(e-1))}_{(X,\Delta)/Y}\otimes\O_{X_{Y^{d(e-1)}}}(R_{Y^{d(e-1)}})\right)_{Y^{de}}\\ &\hspace{70pt}\circ {{f^{(d(e-1))}}_{Y^{de}}}_*\phi^{(d)}_{(X^{d(e-1)},\Delta)/Y^{d(e-1)}}\otimes\left(\L^{(d(e-1))}_{(X,\Delta)/Y}\right)_{Y^{de}}(p^{d(e-1)}R_{Y^{de}}) \\ \cong&({F_Y^{d}}^*\theta^{(d(e-1))})\circ(\theta^{(d)}\otimes\O_{Y^{de}}((p^{d(e-1)}-1)(a^{-1}C-K_{Y^{de}}))) \end{align*}
implies that $E^{(de)}=(p^{d(e-1)}+\cdots+p+1)E^{(d)}=(p^{de}-1)(p^d-1)^{-1}E^{(d)}$ for every $e>0$. We define $\Delta_Y:=(p^d-1)^{-1}E^{(d)}$, this is independent of the choice of $d$ by the above. Note that by this definition \begin{align*}f_*\O_X((p^{d}-1)(K_{X/Y}+\Delta-R))\cong&\O_{Y}((p^{d}-1)(a^{-1}C-K_{Y}))\\\cong&\O_{Y}((p^{d}-1)(\Delta_{Y}-S)),\end{align*} which proves (1).\\ %
\tb{Step2.} We show that for each $e>0$ there exists a commutative diagram $$ \xymatrix { {F_Y^{de}}_*\L^{(de)}_{(Y,\Delta_Y)} \ar[rr]^{(\phi^{(de)}_{(Y,\Delta_Y)}\otimes \O_Y(S))^{**}} \ar[d]_{\cong} && \O_Y(S) \ar[d]^{\cong} & \\ f_*{F_X^{de}}_*\L^{(de)}_{(X,\Delta)}(p^{de}R) \ar[rr]^(0.55){\psi^{(de)}} && f_*\O_X(R) }$$ where $\psi^{(de)}:=f_*((\phi^{(de)}_{(X,\Delta)}\otimes\O_X(R))^{**})$. It is clear that each object of the above diagram is a reflexive sheaf, so we may assume that $Y=Y_0$ and $X$ is smooth. Since $F_X^d=(F_Y^d)_X\circ F_{X/Y}^{(d)}$, we have \begin{align*} \phi^{(d)}_{(X,\Delta)}\otimes\O_X(R):=&\Tr_{F_X^d}\otimes\o_X^{-1}(R) \cong\left(\Tr_{(F_Y^d)_X}\circ {(F_Y^d)_X}_*\Tr_{F_{X/Y}^{(d)}}\right)\otimes\o_X^{-1}(R)\\ \cong&\left(((f^*\Tr_{F_Y})\otimes\o_{X/Y})\circ {(F_Y^d)_X}_*\Tr_{F_{X/Y}^{(d)}}\right)\otimes\o_X^{-1}(R) \\ \cong&(f^*\phi^{(d)}_{Y}\otimes\O_X(R))\circ {(F_Y^d)_X}_*\left(\phi^{(d)}_{(X,\Delta)/Y}\otimes f_{Y^d}^*\o_{Y^d}^{1-p^d}(R_{Y^d})\right). \end{align*} We note that $\phi^{(d)}_Y$ is a morphism between vector bundles on $Y$, thus $\psi^{(d)}$ is decomposed into \begin{align*}\psi^{(d)}\cong &(\phi^{(d)}_{Y}\otimes \O_Y(S)) \circ {F_Y^d}_*(\theta^{(d)}\otimes\o_{Y^d}^{1-p^d}).\end{align*} %
On the other hand, by the definition of $\Delta_Y$, there exists a commutative diagram $$\xymatrix{ \O_Y((1-p^d)\Delta_Y+p^{d}S) \ar[r] \ar[d]_{\cong} & \O_{Y^{d}}(p^{d}S) \ar[d]^{\cong} \\ {f^{(d)}}_*\L^{(d)}_{(X,\Delta)/Y}(p^dR) \ar[r]^{\theta^{(d)}} & {f_{Y^d}}_*\O_{X_{Y^d}}(R_{Y^d}).}$$ Applying the functor ${F_Y^d}_*((\underline{\quad})\otimes\o_Y^{1-p^d})$ to this diagram, we have the following: $$ \xymatrix@R=10mm@C=17mm{ ({F_Y^d}_*\O_Y((1-p^{d})(K_Y+\Delta_Y)))\otimes \O_Y(S) \ar[r] \ar[d]_{\cong} & ({F_Y^d}_*\o_Y^{1-p^{d}})\otimes \O_Y(S) \ar[d]^{\cong} \\ {F_Y^d}_*{f^{(d)}}_*\L^{(d)}_{(X,\Delta)}(p^dR) \ar[r]^{{F_Y^d}_*(\theta^{(d)}\otimes\o_{Y}^{1-p^d})} & ({F_Y^d}_*\o_Y^{1-p^d})\otimes f_*\O_X(R) . } $$ Hence by the decomposition of $\psi^{(d)}$ and the definition of $\phi^{(d)}_{(Y,\Delta_Y)}$, the claim is proved in the case when $e=1$. %
Furthermore, for each $e>0$ we have \begin{align*}\psi^{(d(e+1))} \cong & (f_*(\phi^{(de)}_{(X,\Delta)}\otimes\O_X(R))) \circ {f}_*{F_X^{de}}_*(\phi^{(d)}_{(X^{de},\Delta)}\otimes\L^{(de)}_{(X,\Delta)}(p^{de}R))\\ \cong &\psi^{(de)} \circ {F_Y^{de}}_*{f^{(de)}}_*(\phi^{(d)}_{(X^{de},\Delta)}\otimes\O_{X^{de}}(R+(1-p^{de})(K_{X^{de}}+\Delta-R)))\\ \cong & \phi^{(de)}_{(Y,\Delta_Y)}\otimes\O_Y(S) \circ {F_Y^{de}}_*(\phi^{(d)}_{(Y^{de},\Delta_Y)}\otimes\O_{Y^{de}}(S+(1-p^{de})a^{-1}C))\\ \cong & \phi^{(de)}_{(Y,\Delta_Y)}\otimes\O_Y(S) \circ {F_Y^{de}}_*(\phi^{(d)}_{(Y^{de},\Delta_Y)}\otimes\O_{Y^{de}}((1-p^{de})(K_Y+\Delta_Y)+p^{de}S))\\ \cong & \phi^{(de)}_{(Y,\Delta_Y)}\otimes\O_Y(S) \circ {F_Y^{de}}_*(\phi^{(d)}_{(Y^{de},\Delta_Y)}\otimes\L^{(de)}_{(Y,\Delta_Y)}(p^{de}S)) \cong  \phi^{(d(e+1))}_{(Y,\Delta_Y)} \otimes \O_Y(S). \end{align*} This is our claim. \\ 
\tb{Step3.} We prove statements (2)-(4). (3) is obvious. We show (2). By the definition of $\phi^{(d)}_{(X,\Delta)}$, we may assume that $X$ is smooth. Then there is a commutative diagram $$ \xymatrix@C=40mm{ ({F_Y^{de}}_*{f^{(de)}}_*\L^{(de)}_{(X,\Delta)}(p^{de}R))(D) \ar[r]^{\psi^{(de)}\otimes\O_Y(D)} \ar[d]_{\cong} & (f_*\O_X(R))(D) \ar[d]^{\cong} \\ {F_Y^{de}}_*{f^{(de)}}_*\L^{(de)}_{(X,\Delta)}(p^{de}(f^*D+R)) \ar[r]^(0.6){f_*(\phi^{(de)}_{(X,\Delta)}\otimes\O_Y(f^*D+R))} \ar[d]_{\cong} & f_*\O_X(f^*D+R) \ar[d]^{\cong} \\ {F_Y^{de}}_*{f^{(de)}}_*\L^{(de)}_{(X,\Delta)}(p^{de}(f^*D+B'+R)) \ar[r]^(0.58){f_*(\phi^{(de)}_{(X,\Delta)}\otimes\O_X(f^*D+B'+R))} & f_*\O_X(f^*D+B'+R). }$$
Thus by Step2, \begin{align*} H^0(X,\phi^{(de)}_{(X,\Delta)}\otimes\O_X(f^*D+B'+R))\cong& H^0(Y,f_*(\phi^{(de)}_{(X,\Delta)}\otimes\O_X(f^*D+B'+R)))\\ \cong& H^0(Y,(\phi^{(de)}_{(Y,\Delta_Y)}\otimes(D+S))^{**}),\end{align*}
which implies (2). %
For (4), we may assume that $Y=Y_0$. Then, since $f_*\O_X(R)$ is a line bundle, we only need to show that the case when $f_*\O_X(R)\cong\O_Y$ and $R=\Gamma\ge0$. In this case, since $R$ and $(p^d-1)(K_X+\Delta)$ are Cartier, and since $f$ is flat projective, we have $H^0(X_{y},\O_X(R)|_{X_{y}})\ne0$ and \begin{align*} H^0(X_{y},\O_X((1-p^d)R)|_{X_{y}})=H^0(X_{y},\O_{X_{y}}((1-p^{d})(K_{X_{y}}+\Delta_{y})))\ne0 \end{align*} for every $y\in Y$ by assumptions and upper semicontinuity \cite[Theorem~12.8]{Har77}. In particular, if $X_y$ is reduced then $\O_X(R)|_{X_{y}}\cong\O_{X_{y}}$, because every nonzero endomorphism of a line bundle on a connected reduced projective scheme over a field is an isomorphism. Hence the isomorphism $\O_Y\cong f_*\O_X(R)$ shows that the support of $R$ is contained a union of nonreduced fibers. 
Set $Y_1:=\{y\in Y|H^0(X_y,\O_{X}(R)_y)\cong k(y)\}$. Then we have
\begin{align*}
{\rm Supp\ } \Delta_{Y}|_{Y_1} ={\rm Supp\ } \coker(\theta^{(d)})|_{Y_1}=&\{y\in Y_1|S^0(X_{\ol y},\Delta_{\ol y},\O_{X}(R)|_{X_{\ol y}})=0\},
\end{align*}
where the first (resp. the second) equality follows from the definition of $\Delta_Y$ (resp. Observation \ref{obs:lower semi conti} (I\hspace{-1pt}V)).
Now we prove (a)$\Rightarrow$(b). In the situation of (a) $X_y$ is reduced, and so $y\in Y\setminus f({\rm Supp\ } R)$. 
We recall Example \ref{eg:gl F-sp}, which shows that the global $F$-splitting of $(X_{\ol y},\Delta_{\ol y})$ is equivalent to the equality $S^0(X_{\ol y},\Delta_{\ol y},\O_{X_{\ol y}})=H^0(X_{\ol y},\O_{X_{\ol y}})$. Thus it is enough to show that $y\in Y_1$. 
Let $\ol{\{y\}}$ be the closure in $Y$ of the set $\{y\}$ with the reduced induced subscheme structure. Let $Y'$ be a smooth open subset of $\ol{\{y\}}$ such that $R_{Y'}=0$ and that ${\rm Supp\ }\Delta$ does not contain any irreducible component of any fiber over $Y'$. 
Then for a general closed point $y'\in Y'$, 
\begin{align*}
\dim_{k}S^0(X_{y'},\Delta_{y'},\O_{X_{y'}})=&\dim_{k(\ol y)}S^0(X_{\ol y},\Delta_{y'},\O_{X_{\ol y}})\\
=&\dim_{k(\ol y)}H^0(X_{\ol y},\O_{X_{\ol y}})=\dim_{k}H^0(X_{y'},\O_{X_{y'}}),
\end{align*}
where the first (resp. the third) equality follows from lower semicontinuity proved in Observation \ref{obs:lower semi conti} (I\hspace{-1pt}V) (resp. upper semicontinuity).
Thus $(X_{y'},\Delta_{y'})$ is globally $F$-split, and in particular $X_{y'}$ is reduced. 
Since $k$ is algebraically closed, we have $H^0(X_{y'},\O_{X_{y'}})\cong k$, and hence $H^0(X_y,\O_{X_y})\cong k(y)$, or equivalently, $y\in Y_1$. 
To prove (b)$\Rightarrow$(a), we replace $Y$ by its affine open subset contained in $Y\setminus({\rm Supp\ }\Delta_Y\cup f({\rm Supp\ }R))$. Then the surjectivity of $\theta^{(d)}$ shows that $\phi^{(d)}_{(X,\Delta)/Y}:{F^{(d)}_{X/Y}}_*\L^{(d)}_{(X,\Delta)/Y}\to \O_{X_{Y^{d}}}$ is split, and thus so is $\phi^{(d)}_{(X,\Delta)/Y}|_{X_{\ol y^d}}:{F^{(d)}_{X_{\ol y}/{\ol y}}}_*(\L^{(d)}_{(X,\Delta)/Y}|_{(X_{\ol y})^d})\to\O_{X_{\ol y^{d}}}$. This means that $\Delta$ does not contain any irreducible component of $f^{-1}(y)$, so $\Delta_{\ol y}$ is well-defined, and we have $\phi^{(d)}_{(X,\Delta)/Y}|_{X_{\ol y^d}}\cong\phi^{(d)}_{(X_{\ol y},\Delta_{\ol y})/\ol y}$, which completes the proof.  \end{proof}
\begin{proof}[Proof of Corollary $\ref{cor:Iitaka fibrat}$]
We use the notation of Observation \ref{obs:Iitaka fibrat}. Let $l>0$ be an integer such that $l(K_X+\Delta)$ is Cartier and base point free. We replace $Y$ by its smooth locus $Y_{\rm sm}$, and $X$ by the smooth locus of $f^{-1}(Y_{\rm sm})$. As in the proof of Theorem \ref{thm:quasi can bdl formula}, we set $$\psi^{(e)}:=f_*(\phi^{(e)}_{(X,\Delta)}\otimes\O_X(R))\textup{~and~} \theta^{(e)}:={f_{Y^e}}_*(\phi^{(e)}_{(X,\Delta)/Y}\otimes\O_{X}(R)_{Y^e})$$ for every $e>0$ divisible enough.  Since $S^0(X,\Delta,(p^d-1)(K_X+\Delta))\ne0$, we have \begin{align*} 0\ne &S^0(X,\Delta,(l-1)(p^d-1)(K_X+\Delta))\\ \hookrightarrow &S^0(X,\Delta,(l-1)(p^d-1)(K_X+\Delta)+B)=S^0(X,\Delta,f^*l(p^d-1)H+R).  \end{align*} 
This implies the morphism $$f_*(\phi^{(e)}_{(X,\Delta)}\otimes\O_X(f^*l(p^d-1)H+R))\cong \psi^{(e)}\otimes\O_Y(l(p^d-1)H)$$ is nonzero for an $e>0$ divisible enough, where the isomorphism follows from projection formula. Hence $\psi^{(e)}$ is also nonzero. By an argument similar to Step2, we can factor $\psi^{(e)}$ into $(\phi^{(e)}_{Y}\otimes\O_Y(S))\circ {F_Y^e}_*(\theta^{(e)}\otimes\o_Y^{1-p^e})$, and hence $\theta^{(e)}$ is nonzero. Thus ${\theta^{(e)}}\otimes k({\ol\eta})\cong H^0(X_{\ol\eta},\phi^{(e)}_{(X_{\ol\eta}/\ol\eta,\Delta_{\ol\eta})})$ is nonzero, or equivalently, $(X_{\ol\eta},\Delta_{\ol\eta})$ is globally $F$-split. In particular $X_{\ol\eta}$ is reduced, and this means that $f$ is separable. 
We show (2). First note that $R/R_S(X,\Delta,K_X+\Delta)$ is finite dimensional if and only if so is $R/R_S(X,\Delta,l(K_X+\Delta))$ by Lemma \ref{lem:fg}. 
Let $m\ge0$ be an integer with $l|m$. Then we have $H^0(X,lm(K_X+\Delta))\cong H^0(Y,lm(K_Y+\Delta))$. Furthermore, by Theorem \ref{thm:quasi can bdl formula} (2), we have 
\begin{align*} S^0(X,\Delta,m(K_X+\Delta))&=S^0(X,\Delta,f^*(mH-S)+R) \\ &=S^0(Y,\Delta_Y,mH-S+S)=S^0(Y,\Delta_Y,mH). \end{align*}
Thus $R/R_S(X,\Delta,l(K_X+\Delta))\cong R/R_S(Y,\Delta_Y,lH)$.  By Example \ref{eg:rs for ample div}, this $k$-vector space is finite dimensional if and only if $(Y,\Delta_Y)$ is $F$-pure, which is our claim.
\end{proof}
\begin{eg}\label{eg:ell surface} 
Let $f:X\to Y$ be a relatively minimal elliptic fibration. In other words, let $f$ be a generically smooth fibration from a smooth projective surface $X$ to a smooth projective curve $Y$, whose fibers have arithmetic genus one and do not contain $(-1)$-curves of $X$. Then by the canonical bundle formula \cite[Theorem~2]{BM77}, we have $$K_X\sim f^*D+\sum_{i=1}^rl_iF_i,$$ where $D$ is a divisor on $Y$, $m_iF_i=X_{y_i}$ is a multiple fiber with the multiplicity $m_i$, and $0\le l_i<m_i$. 
Let $m$ be the least common multiple of $m_1,\ldots,m_r$, and let $a,e\ge0$ be integers such that $m=ap^e$ and $p\nmid a$. 
We set $$R:=\sum_{i-1}^r\left\{\frac{(1-p^d)l_i}{m_i}\right\}m_iF_i$$ for some $d\ge e$ satisfying $a|(p^d-1)$. 
Here, recall that for every $s\in\Q$, $\{s\}$ is the fractional part $s-\lfloor s\rfloor$ of $s$.
It is easily seen that $f_*\O_X(R)\cong\O_Y$ and 
\begin{align*}
aK_X-aR\sim& af^*D+a\sum_{i=1}^rl_iF_i-a\sum_{i=1}^r(1-p^d)l_iF_i+a\sum_{i=1}^r \lfloor\frac{(1-p^d)l_i}{m_i}\rfloor m_iF_i \\
=& f^*(aD+\sum_{i=1}^r (\frac{al_ip^{d}}{m_i}+a\lfloor\frac{(1-p^d)l_i}{m_i}\rfloor)y_i).
\end{align*} Thus, $a$ and $R$ satisfy condition (ii) of Theorem \ref{thm:quasi can bdl formula}. Furthermore, assume that the geometric generic fiber of $f$ is globally $F$-split, or equivalently, is an elliptic curve with nonzero Hasse invariant. Then by Theorem \ref{thm:quasi can bdl formula} there exists an effective $\Z_{(p)}$-divisor $\Delta_Y$ on $Y$ such that $$S^0(X,f^*D'+R)=S^0(Y,\Delta_Y,D')$$ for every divisor $D'$ on $Y$, and $y_1,\ldots,y_r\in f(R)\cup \Delta_Y$. Remark that if $p\nmid m_i$ for each $i$, then $m_i|(p^d-1)$, and so $R=0$. 
\end{eg} 
Finally, applying Theorem \ref{thm:quasi can bdl formula}, we show that for a smooth projective surface $X$ of general type, $R/R_S(X,K_X)$ is finite dimensional (Corollary \ref{cor:surf of gen type}).
\begin{cor}\label{cor:birat mor} \samepage Let $f:X\to Y$ be a birational morphism between normal projective varieties, let $\Delta$ be an effective $\Q$-Weil divisor on $X$, and let $\Delta_Y:=f_*\Delta$. Assume that $a(K_Y+\Delta_Y)$ is Cartier for some $a>0$ not divisible by $p$ and $(Y,\Delta_Y)$ is canonical. Then for each $m>0$, $$ S^0(Y,\Delta_Y,am(K_Y+\Delta_Y)) \cong S^0(X,\Delta,am(K_X+\Delta)).$$ 
\end{cor}
\begin{proof} Since $(Y,\Delta_Y)$ is canonical, $R:=a(K_X+\Delta)-f^*a(K_Y+\Delta_Y)$ is an effective Weil divisor on $X$ supported on the exceptional locus of $f$. Note that $f_*\O_X(R)\cong\O_Y$. We set $B:=(a-1)R$ and $B'_m:=(m-1)R$ for each $m\ge1$. Then we have $$ aR=R+(a-1)R=a(K_X+\Delta)-f^*a(K_Y+\Delta_Y)+B.$$ Thus, by Theorem \ref{thm:quasi can bdl formula}, we have \begin{align*}S^0(Y,\Delta_Y,am(K_Y+\Delta_Y))&\cong S^0(X,\Delta,B'_m+f^*am(K_Y+\Delta_Y)+R)\\&\cong S^0(X,\Delta,am(K_X+\Delta)).\end{align*}  \end{proof}
\begin{cor}[\textup{\cite[Excercise~5.15]{PST14}}] \label{cor:birat inv} \samepage Let $\varphi:Y\dashrightarrow Y'$ be a birational map between normal projective varieties, let $\Delta$ be an effective $\Q$-Weil divisor on $Y$, and let $\Delta':=\varphi_*\Delta$. Assume that $a(K_Y+\Delta)$ and $a(K_{Y'}+\Delta')$ are Cartier for some $a>0$ not divisible by $p$, and that $(Y,\Delta)$ and $(Y',\Delta')$ are canonical. Then, for each $m>0$, $$S^0(Y,\Delta,am(K_Y+\Delta))\cong S^0(K_{Y'},\Delta',am(K_{Y'}+\Delta')).$$  \end{cor}
\begin{proof} This follows directly from Corollary \ref{cor:birat mor}. \end{proof}
The following corollary will be used to prove weak positivity theorem when geometric generic fibers are normal projective surfaces of general type with rational double point singularities (Corollary \ref{cor:2 dim}). Recall that the finiteness of the dimension of $R/R_S$ is equivalent to the assumption (ii) of the main theorem (Theorems \ref{thm:main thm intro} or \ref{thm:main thm}).
\begin{cor}\label{cor:surf of gen type} \samepage Let $X$ be a normal projective surface of general type with rational double point singularities. If $p\ge 7$, then $R/R_S(X,K_X)$ is a finite dimensional vector space.\end{cor}
\begin{proof} By Corollary \ref{cor:birat inv}, we may assume that $X$ is a smooth projective surface of general type which has no $(-1)$-curve. Then for each $n\gg0$, $nK_X$ is base point free, and $Y:=\Proj R(X,K_X)$ has only rational double point singularities \cite[Theorem~9.1]{Bad01}. When $p\ge7$, $Y$ is $F$-pure, because of the classification of rational double points \cite[Section 3]{Art77}, and of Fedder's criterion \cite{Fed83}. Hence the statement follows from Corollary \ref{cor:Iitaka fibrat}. \end{proof}
\section{Weak positivity and numerical invariant}\label{section:inv}
In this section, we define an invariant of coherent sheaves on normal varieties which measures positivity. This will play an important role in the proof of the main theorem. \par We first recall some definitions. %
\begin{defn} \samepage A coherent sheaf $\G$ on a variety $Y$ is said to be {\it generically globally generated} if the natural morphism $H^0(Y,\G)\otimes_k\O_Y\to\G$ is surjective on the generic point of Y. \end{defn}
Viehweg introduced the notion of weak positivity as a generalization of nefness of vector bundles. %
\begin{defn}[\textup{\cite[Notation, (vii)]{Kol87}}] \label{defn:wp} \samepage A coherent sheaf $\G$ on a normal projective variety $Y$ is said to be {\it weakly positive} if for every ample divisor $H$ and for each $\a>0$ there exist some $\b>0$ such that $(S^{\a\b}\G)^{**}\otimes\O_Y(\b H)$ is generically globally generated. Here $(S^{\a\b}\G)^{**}$ is the double dual of the $\a\b$-th symmetric product of $\G$. \end{defn}
\begin{rem} \label{rem:wp} (1) The above definition is independent of the choice of $H$ (e.g.,\textup{\cite[Lemma 2.14]{Vie95}}). \\ (2) $\G$ is weakly positive if and only if its double dual $\G^{**}$ is weakly positive. \\ (3) Assume that $\G$ is a vector bundle on a smooth projective curve. Then $\G$ is weakly positive if and only if it is nef.  \end{rem}
\begin{defn} \label{defn:inv} \samepage Let $Y$ be a normal variety, let $\G$ be a coherent sheaf, and let $H$ be an ample Cartier divisor. Then we define  \begin{align*} T(Y,\G,H)&:= \begin{array}{cc|ccl}
\ldelim\{{3}{0.1pt}&& \textup{there exists an integer $e>0$ such that}& \rdelim\}{3}{1pt} \cr
&\e\in\Q&\textup{ $p^e\e\in\Z$ and $({F_Y^e}^*\G)(-p^e\e H)$ is}& \cr
&                 &\textup{generically globally generated.}& 
\end{array}~,\\
t(Y,\G,H)&:=\sup T(Y,\G,H)\in \R\cup\{\infty\}.  \end{align*} \end{defn}
\begin{lem} \label{lem:t} \samepage Let $Y,\G,H$ be as in Definition $\ref{defn:inv}$, and let $\F$ be a coherent sheaf.  \begin{itemize} \item[$(1)$]If there exists a generically surjective morphism $\F\to\G$, then $t(Y,\F,H)\le t(Y,\G,H)$.  \item[$(2)$]$ t(Y,\F,H)+t(Y,\G,H) \le t(Y,\F\otimes\G,H) $.  \item[$(3)$]For each $e>0$, $t(Y,{F_Y^e}^*\G,H)=p^et(Y,\G,H)$.  \end{itemize} \end{lem}
\begin{proof} This follows directly from the definition. \end{proof}
\begin{prop}\label{prop:t ggg} \samepage 
Let $Y$ be a normal projective variety of dimension $n$, let $\G$ and $H$ be as in Definition $\ref{defn:inv}$, and let $\Delta$ be an effective $\Q$-Weil divisor on $Y$ such that $K_Y+\Delta$ is $\Q$-Cartier. 
Set $t:=t(Y_0,\G|_{Y_0},H|_{Y_0})$ for an open subset $Y_0\subseteq Y$ satisfying $\codim(Y\setminus Y_o)\ge 2$. 
If $D$ is a Cartier divisor such that $D-(K_Y+\Delta)-nA+tH$ is ample for some base point free ample divisor $A$, then $\G^{**}(D)$ is generically globally generated. \end{prop}
\begin{proof} 
Since $t(Y_0,\G|_{Y_0},H|_{Y_0})$ is the supremum, there exists an $\e\in T(Y_0,\G|_{Y_0},H|_{Y_0})$ such that $B:=D-(K_Y+\Delta)-nA+\e H$ is ample. We fix such an $\e$. By the definition, $p^e\e\in\Z$ and there is a generically surjective morphism $\bigoplus\O_Y\to({F_Y^e}^*\G)^{**}(-p^e\e H)$ for every $e\gg0$.
Let $l>0$ be an integer such that $l(K_Y+\Delta)$ is Cartier and $l\e\in\Z$. For every $e\ge0$, we denote by $q_e$ and $r_e$ respectively the quotient and the remainder of the division of $p^e-1$ by $l$.
Hence we have following generically surjective morphisms 
\begin{align*}
&\bigoplus{F_Y^e}_*\O_Y(q_el(B+nA)+(r_e+1)(D-K_Y+\e H)+K_Y) \\
\cong &{F_Y^e}_*\bigoplus\O_Y(q_el(B+nA)+(r_e+1)(D-K_Y+\e H)+K_Y) \\
\to & {F_Y^e}_*(({F_Y^e}^*\G)^{**}\otimes\O_Y(q_el(B+nA-\e H)+(r_e+1)(D-K_Y)+K_Y)) \\
= & {F_Y^e}_*(({F_Y^e}^*\G)^{**}\otimes\O_Y(-q_el\Delta+p^e(D-K_Y)+K_Y)) \\
\cong & \left(\G^{**}(D)\otimes{F_Y^e}_*\O_Y((1-p^e)K_Y-q_el\Delta)\right)^{**}\\
\to & \left(\G^{**}(D)\otimes{F_Y^e}_*\O_Y((1-p^e)K_Y)\right)^{**} \\
\to &\G^{**}(D),
\end{align*}
where the isomorphism in the fifth line is induced by the projection formula, and the last morphism is induced by $\phi^{(e)}_{Y}$. Therefore, it is sufficient to show that 
$${F_Y^e}_*\O_Y(q_el(B+nA)+(r_e+1)(D-K_Y+\e H)+K_Y)$$
is globally generated for each $e\gg0$. Since $0\le r_e<l$, by the Serre vanishing theorem, we have 
\begin{align*} 
&H^i(Y,\O_Y(-iA)\otimes{F_Y^e}_*\O_Y(q_el(B+nA)+(r_e+1)(D-K_Y+\e H)+K_Y)) \\
\cong &H^i(Y,{F_Y^e}_*\O_Y(q_el(B+(n-i)A)+(r_e+1)(D-K_Y+\e H-iA)+K_Y)=0
\end{align*} for each $i>0$ and $e\gg0$. Hence our claim follows from the Castelnuovo-Mumford regularity (\cite[Theorem 1.8.5]{Laz04}). \end{proof}
\begin{prop}\label{prop:t and wp}\samepage Let $Y$ be a normal projective variety, let $\G$ be a coherent sheaf, and let $H$ be an ample divisor. If $t(Y_0,\G|_{Y_0},H|_{Y_0})\ge0$ for some open subset $Y_0 \subseteq Y$ satisfying $\codim(Y\setminus Y_0)\ge 2$, then $\G$ is weakly positive. \end{prop}
\begin{proof} By the hypothesis and Lemma \ref{lem:t}, we have $$t(Y_0,(S^{\a}\G)^{**}|_{Y_0},H|_{Y_0})\ge t(Y_0,\G^{\otimes \alpha}|_{Y_0},H|_{Y_0})\ge\alpha t(Y_0,\G|_{Y_0},H|_{Y_0})\ge0$$ for every $\a>0$. Applying the previous proposition, we obtain an ample Cartier divisor $D$ such that $(S^{\a}\G)^{**}(D)$ is generically globally generated for every $\a>0$, which is our claim. \end{proof}
The following proposition will be used in Section \ref{section:s-st}.
\begin{prop}\label{prop:t on curve} \samepage Let $Y$ be a smooth projective curve, let $\G$ be a vector bundle on $Y$, and let $H$ be an ample divisor on $Y$. If $t(Y,\G,H)\ge 0$, then $\G$ is nef. Moreover, if $t(Y,\G,H)>0$, then $\G$ is ample. \end{prop}
\begin{proof} The first statement follows directly from Proposition \ref{prop:t and wp} and Remark \ref{rem:wp}. We prove the second statement. By the hypothesis, $({F_Y^e}^*\G)(-H)$ is generically globally generated for some $e>0$, so it is a nef vector bundle since $Y$ is a curve. This means that ${F_Y^e}^*\G$ is an ample vector bundle, hence so is $\G$. \end{proof}
\section{Main theorem}\label{section:thm}
In this section, we prove the main theorem (Theorem \ref{thm:main thm}). 
As applications of the theorem, we prove weak positivity theorem for certain surjective morphisms of relative dimension zero, one and two (Corollaries \ref{cor:0 dim}, \ref{cor:1 dim} and \ref{cor:2 dim} respectively).
\begin{thm}\label{thm:main thm}\samepage
Let $f:X\to Y$ be a separable surjective morphism between normal projective varieties, let $\Delta$ be an effective $\Q$-Weil divisor on $X$ such that $a\Delta$ is integral for some integer $a>0$ not divisible by $p$, and let $\ol \eta$ be the geometric generic point of $Y$. Let $H$ be an ample Cartier divisor on $Y$. Assume that \begin{itemize} \item[(i)] $K_{X_{\ol\eta}}+\Delta_{\ol\eta}$ is finitely generated in the sense of Definition \ref{defn:fg}, and \item[(ii)] there exists an integer $m_0>0$ such that $$S^0(X_{\ol\eta},\Delta_{\ol \eta},am(K_{X_{\ol \eta}}+\Delta)_{\ol \eta})=H^0(X_{\ol\eta},am(K_{X_{\ol \eta}}+\Delta)_{\ol \eta})$$ for each $m\ge m_0$.  \end{itemize} Then 
$f_*\O_X(am(K_{X}+\Delta))\otimes\o_Y^{-am}$ is a weakly positive sheaf for every $m\ge m_0$. %
\end{thm}
\begin{proof} We first note that $X_{\ol\eta}$ is a $k(\ol\eta)$-scheme of pure dimension satisfying $S_2$ and $G_1$, and that each connected component of $X_{\ol\eta}$ is integral by the Stein factorization and separability of $K(X)/K(Y)$.  Let $d>0$ be an integer satisfying $a|(p^d-1)$. \\
\tb{Step1.} In this step we reduce to the case where $X$ and $Y$ are smooth. Let $H$ be an ample Cartier divisor on $Y$. By Proposition \ref{prop:t and wp}, it suffice to prove that $t(Y_{0},(f_*\O_X(am(K_X+\Delta)))|_{Y_{0}}\otimes\o_{Y_{0}}^{-am},H|_{Y_{0}})\ge0$ for each $m\ge m_0$, where $Y_{0}\subseteq Y$ is an open subset satisfying $\codim(Y\setminus Y_0)\ge2$. Hence, replacing $X$ and $Y$ by their smooth loci, we may assume that $f$ is a dominant morphism between smooth varieties (the projectivity of $f$ may be lost, but we will not use it). \par
We set $t(m):=t(Y,f_*\O_X(am(K_{X/Y}+\Delta)),H)$ for each $m>0$. \\ 
\tb{Step2.} We show that there exist integers $l,n_0>m_0$ such that $t(l)+t(n)\le t(l+n)$ for each $n\ge n_0$. By the hypothesis (i) and Lemma \ref{lem:fg}, $R(X_{\ol\eta},a(K_{X_{\ol\eta}}+\Delta)_{\ol\eta})$ is a finitely generated $k(\ol\eta)$-algebra. Hence for every $l>m_0$ divisible enough there exists an $n_0>m_0$ such that the natural morphism $$H^0(X_{\ol \eta},al(K_{X_{\ol \eta}}+\Delta)_{\ol\eta})\otimes H^0(X_{\ol\eta},an(K_{X_{\ol\eta}}+\Delta)_{\ol\eta})\to H^0(X_{\ol\eta},a(l+n)(K_{X_{\ol\eta}}+\Delta)_{\ol\eta})$$ is surjective. This shows that the natural morphism $$f_*\O_X(al(K_{X/Y}+\Delta))\otimes f_*\O_X(an(K_{X/Y}+\Delta))\to f_*\O_X(a(l+n)(K_{X/Y}+\Delta))$$ is generically surjective, thus we have $ t(l)+t(n)\le t(l+n)$ by Lemma \ref{lem:t}. \\
\tb{Step3.} We show that $t(mp^{de}-a^{-1}(p^{de}-1))\le p^{de}t(m)$ for each $e>0$ and for each $m\ge m_0$. By the hypothesis (ii) and Observation \ref{obs:lower semi conti} (I\hspace{-1pt}I\hspace{-1pt}I), there exist generically surjective morphisms \begin{align*} {f^{(de)}}_*\O_{X^{de}}(((am-1)p^{de}+1)(K_{X^{de}/Y^{de}}+\Delta))\xrightarrow{\alpha} &{f_{Y^{de}}}_*\O_{X_{Y^{de}}}(am(K_{X_{Y^{de}}/Y^{de}}+\Delta_{Y^{de}}))\\\cong&{F_Y^{de}}^*f_*\O_X(am(K_{X/Y}+\Delta)),\end{align*} 
where $\alpha:={f_{Y^{de}}}_*(\phi^{(de)}_{(X,\Delta)/Y}\otimes\O_{X_{Y^{de}}}(am(K_{X_{Y^{de}}/Y^{de}}+\Delta_{Y^{de}})))$, and the isomorphism follows from the flatness of $F_Y$.
Hence by Lemma \ref{lem:t}, we have $$t(mp^{de}-a^{-1}(p^{de}-1))\le t(Y,{F_Y^{de}}^*f_*(\o_{X/Y}^{am}(am\Delta)),H)=p^{de}t(m).$$
\tb{Step4.} We prove the theorem. Set $m\ge m_0$. If $am=1$, then $t(1)\le p^{d}t(1)$ by Step3, which gives $t(1)\ge0$. Thus we may assume $am_0\ge2$. Let $q_{m,e}$ be the quotient of $mp^{de}-a^{-1}(p^{de}-1)-n_0$ by $l$ and let $r_{m,e}$ be the remainder for $e\gg0$. We note that $q_{m,e}>0$ since $m\ge m_0\ge2a^{-1}>a^{-1}$, and that $p^{de}-q_{l,e}\xrightarrow{e\to\infty}\infty$. Then \begin{align*}&q_{m,e}t(l)+t(r_{m,e}+n_0)\overset{Step2}{\le}t(mp^{de}-a^{-1}(p^{de}-1))\overset{Step3}{\le}p^{de}t(m),\end{align*} and so $c:=\min\{t(r+n_0)|0\le r<l\}\le p^{de}t(m)-q_{m,e}t(l).$ By substituting $l$ for $m$, we have $c\le (p^{de}-q_{l,e})t(l)$ for each $e\gg0$, which means $t(l)\ge0$. Hence $c\le p^{de}t(m)$ for each $e\gg0$, and consequently $t(m)\ge0$. This completes the proof.  \end{proof}
\begin{rem} \label{rem:Raynaud-Xie} There exists a fibration $g:S\to C$ from a smooth projective surface $S$ to a smooth projective curve $C$ such that $g_*\o_{S/C}^m$ is not nef for any $m>0$ \cite{Ray78}\cite[Theorem~3.6]{Xie10}. This fibration does not satisfy condition (ii) of Theorem \ref{thm:main thm}. Indeed, the geometric generic fiber of $g$ is a Gorenstein curve which has a cusp, hence by \cite{GW77} it is not $F$-pure. Since the dualizing sheaf of a Gorenstein curve not isomorphic to $\mathbb P^1$ is trivial or ample, the claim follows from Examples \ref{eg:gl F-sp} and \ref{eg:rs for ample div}. \end{rem}
\begin{cor}\label{cor:0 dim}\samepage Let $f:X\to Y$ be a surjective morphism between normal projective varieties, let $\Delta$ be an effective $\Q$-Weil divisor on $X$, and let $a>0$ be an integer such that $a\Delta$ is integral. 
If $f$ is separable and generically finite, then $f_*\O_X(a(K_{X}+\Delta))\otimes\o_Y^{-a}$ is weakly positive.  \end{cor}
\begin{proof} 
Since $f$ is generically finite, the natural morphism $$f_*\O_X(aK_X)\otimes\o_Y^{-a}\to f_*\O_X(a(K_X+\Delta))\otimes\o_Y^{-a}$$ is an isomorphism at the generic point of $Y$. Thus it is enough to show the case of $\Delta=0$.
Since the geometric generic fiber $X_{\ol\eta}$ is a reduced $k(\ol\eta)$-scheme of dimension zero, the assertion follows directly from Theorem \ref{thm:main thm}.
\end{proof}
\begin{cor}\label{cor:1 dim}\samepage
Let $f:X\to Y$ be a separable surjective morphism of relative dimension one between normal projective varieties, let $\Delta$ be an effective $\Q$-Weil divisor on $X$, and let $a>0$ be an integer not divisible by $p$ such that $a\Delta$ is integral. 
Assume that $(X_{\ol\eta},\Delta_{\ol\eta})$ is $F$-pure, where $\ol\eta$ is the geometric generic point of $Y$. If $K_{X_{\ol\eta}}+\Delta_{\ol\eta}$ is ample $($resp. $K_{X_{\ol\eta}}$ is ample and $a\ge2$$)$, then $f_*\O_X(am(K_X+\Delta))\otimes\o_Y^{-am}$ is weakly positive for each $m\ge2$ $($resp. $m\ge1$$)$. 
In particular, if every connected component of $X_{\ol\eta}$ is a smooth curve of genus at least two, then $f_*\o_X^m\otimes\o_Y^{-m}$ is weakly positive for each $m\ge2$.
\end{cor}
\begin{proof} 
Let $U\subseteq X$ be a Gorenstein open subset such that $\codim(X\setminus U)\ge2$ and that $a\Delta|_{U}$ is Cartier. Since $\dim (X\setminus U)\le\dim X-2=\dim Y-1$, $X\setminus U$ does not dominate $Y$.
Thus there exists an open subset $Y_0\subseteq Y$ such that $f|_{X_0}:X_0\to Y_0$ is a Gorenstein morphism and that $a\Delta|_{X_0}$ is Cartier, where $X_0:=f^{-1}(Y_0)$. In particular, $X_{\ol\eta}$ is Gorenstein and $(a\Delta)_{\ol\eta}$ is Cartier. 
Thus the statement follows directly from Corollary \ref{cor:rs for F-pure curve} and Theorem \ref{thm:main thm}. 
\end{proof}
\begin{cor}\label{cor:2 dim}\samepage
Let $f:X\to Y$ be a separable surjective morphism of relative dimension two between normal projective varieties. Assume that every connected component of geometric generic fiber is a normal surface of general type with rational double point singularities, and $p\ge7$. Then $f_*\o_{X}^m\otimes\o_Y^{-m}$ is weakly positive sheaf for each $m\gg0$. \end{cor}
\begin{proof} We note that in this case $K_{X_{\ol\eta}}$ is finitely generated (cf. \cite[Corollary~9.10]{Bad01}). This follows from Corollary \ref{cor:surf of gen type} and Theorem \ref{thm:main thm}. \end{proof}
\section{Semi-stable fibration} \label{section:s-st}
In this section, we discuss the weak positivity theorem in the case of a semi-stable fibration. We begin by recalling some definitions. 
\begin{defn} \samepage A projective $k$-scheme $C$ of dimension one is said to be {\it minimally semi-stable} if: \begin{itemize} \item it is reduced and connected,  \item all the singular points are ordinary double point, and \item each irreducible component which is isomorphic to $\mb{P}^1$ meets other components in at least two points.  \end{itemize} \end{defn}
\begin{defn} \samepage Let $X$ be a smooth projective surface, and let $Y$ be a smooth projective curve. A fibration $f:X\to Y$ is called {\it semi-stable fibration} if all the fibers are minimally semi-stable. \end{defn}
\begin{defn} \samepage A fibration $f:X\to Y$ is said to be {\it isotrivial} if for every two closed fibers are isomorphic. \end{defn}
\begin{notation} \label{notation:s-st} \samepage Let $f:X\to Y$ be a semi-stable fibration whose geometric generic fiber $X_{\ol\eta}$ is a smooth curve of genus at least two.  \end{notation}
In the situation of Notation $\ref{notation:s-st}$, by \cite[4.3~Theorem]{Kol90} or by Corollary \ref{cor:1 dim}, $f_*\o_{X/Y}^m$ is a nef vector bundle for each $m\ge 2$. In the first part of this section, we give a necessary and sufficient condition in terms of $f$ for these vector bundles to be ample (Theorem \ref{thm:ample}). In the second part, we consider the positivity of $f_*\o_{X/Y}$. 
\begin{eg} \label{eg:isotrivial} In the situation of \ref{notation:s-st}, assume that $f$ is isotrivial. Then there exists a finite morphism $\vp:Y'\to Y$ from a smooth projective curve $Y'$ such that there exists a commutative diagram  \[\xymatrix{ Z \ar[d] & Z\times_k Y' \ar[l]_(0.6)g \ar[r]^(0.6){\cong} \ar[d]^h & X_{Y'} \ar[r] \ar[d]^{f_{Y'}} & X \ar[d]^f \\ k & Y' \ar@{=}[r] \ar[l] & Y' \ar[r]^{\vp} & Y, } \] where $Z$ is a closed fiber of $f$ \cite{Szp79}. Let $D$ be an effective divisor on $X$ such that $(K_{X/Y}\cdot D)=0$. Then since $\o_{Z\times_kY'/Y'}\cong g^*\o_Z$, and since $\o_Z$ is ample, each irreducible component of $D_{Y'}$ is a fiber of $g$, thus we have $D_{Y'}\sim g^*E$ for some effective divisor $E$ on $Z$. Then for each integer $m\ge1$, we have $$ \vp^*f_*\o_{X/Y}^m(D) \cong h_*\o_{(Z\times_k Y')/Y'}^m(D_{Y'}) \cong h_*g^*\o_{Z}^m(E) \cong H^0(Z,\o_Z^m(E))\otimes_k \O_{Y'}, $$ and the right-hand side is trivial, so is $\vp^*f_*\o_{X/Y}^m(D)$. In particular, $f_*\o_{X/Y}^m(D)$ is not ample for each $m\ge1$. \end{eg}
In order to prove Theorem \ref{thm:ample}, we recall some results due to Szpiro and due to Tanaka. 
\begin{thm}[\textup{\cite[Th\'eor\`em~1]{Szp79}}] \label{thm:Szp} \samepage
In the situation of Notation $\ref{notation:s-st}$, assume that $f$ is non-isotrivial. Then $\o_{X/Y}$ is nef and big. Furthermore, an integral curve $C$ in $X$ satisfies $(\o_{X/Y}\cdot C)=0$ if and only if $C$ is a smooth rational curve with $(C^2)=-2$ contained in a fiber. \end{thm}
\begin{thm}[\textup{\cite[Theorem~2.6]{Tan12}}] \label{thm:wKV} \samepage
Let $X$ be a smooth projective surface, let $B$ be a nef big $\R$-divisor whose fractional part is simple normal crossing, and let $N$ be a nef $\R$-divisor which is not numerically trivial. Then there exists a real number $r(B,N)>0$ such that \[ H^i(X,K_X+\lceil B\rceil+rN+N')=0 \] for each $i>0$, for every real number $r\ge r(B,N)$, and for every nef $\R$-divisor $N'$ such that $rN+N'$ is a $\Z$-divisor. \end{thm}
\begin{thm}\label{thm:ample} \samepage In the situation of Notation $\ref{notation:s-st}$, let $\Delta$ be an effective $\Z_{(p)}$-divisor on $X$. Assume that $\lfloor\Delta_{\ol\eta}\rfloor=0$. Then the following conditions are equivalent: \begin{itemize} \item[(1)] $f$ is non-isotrivial or $(K_{X/Y}\cdot\Delta)>0$. \item[(2)] $f_*\O_X(m(K_{X/Y}+\Delta))$ is ample for each $m\ge 2$ such that $m\Delta$ is integral. \item[(3)] $f_*\O_X(m(K_{X/Y}+\Delta))$ is ample for some $m\ge 2$ such that $m\Delta$ is integral. \end{itemize} In particular, $f_*\o_{X/Y}^m$ is ample for each $m\ge2$ if and only if $f$ is non-isotrivial. \end{thm}
\begin{proof} $(2)\Rightarrow(3)$ is obvious. $(3)\Rightarrow(1)$ is follows from Example \ref{eg:isotrivial}, thus we have only to prove $(1)\Rightarrow(2)$. \\ \tb{Step1.} We show that if $f$ is isotrivial, then there exists an $\e_1\in(0,1)\cap\Q$ such that $K_{X/Y}+\e\Delta$ is ample for every $\e\in(0,\e_1)\cap\Q$. We may assume that $X\cong Z\times_kY$ for a smooth projective curve $Z$ of genus $\ge2$, and $f:X\to Y$ is the second projection. By the assumption of (1), there is an irreducible component $\Delta_i$ of $\Delta$ such that $g(\Delta_i)=Z$, where $g:X\to Z$ be the first projection. Thus there exists an $\e_1\in(0,1)\cap\Q$, $(K_{X/Y}+\e\Delta\cdot C)>0$ for every $\e\in(0,\e_1)\cap\Q$ and for every integral curve $C$ in $X$. This means that $K_{X/Y}+\e\Delta$ is an ample $\Q$-divisor by the Nakai-Moishezon criterion. \\
\tb{Step2.} Let $H$ be an ample divisor on $Y$. We show that there exists an $\e_2\in[0,\e_1)\cap\Q$ satisfying $H^1(X,m(K_{X/Y}+\e_2\Delta)-f^*H)=0$ for each divisible enough $m>0$. When $f$ is isotrivial, this follow from Step 1 and the Serre vanishing theorem. We assume that $f$ is non-isotrivial, and set $\e_2:=0$. Since $K_{X/Y}$ is nef and big by Theorem \ref{thm:Szp}, there exists an integer $n>0$ such that $nK_{X/Y}-f^*(K_Y+H)\sim B$ for some effective divisor $B$. Let $C$ be an integral curve in $X$ satisfying $(B\cdot C)<0$. Then we have $$(f^*(K_Y+H)\cdot C)>(B+f^*(K_Y+H)\cdot C)=(nK_{X/Y}\cdot C)\ge0,$$ so $C$ dominates $Y$, hence $(K_{X/Y}\cdot C)>0$ by Theorem \ref{thm:Szp}. Replacing $n$ by a larger one if necessary, we may assume that $B$ is also nef and big. Then for every $m\gg0$, we have $$H^1(X,mK_{X/Y}-f^*H)=H^1(X,K_X+B+(m-n-1)K_{X/Y})=0,$$ where the second equality follows from Theorem \ref{thm:wKV}. \\
\tb{Step3.} We show that $t(Y,f_*\O_X(m(K_{X/Y}+\e_2\Delta)),H)>0$ for each divisible enough $m>0$. There exists an exact sequence \begin{align*}0\to\O_X(m(K_{X/Y}+\e_2\Delta)-f^*H-X_y) \to &\O_X(m(K_{X/Y}+\e_2\Delta)-f^*H) \\\xrightarrow{\rho} &\O_{X_y}(m(K_{X/Y}+\e_2\Delta)-f^*H) \to 0 \end{align*} for every closed point $y\in Y$. By Step 2, since $H+y$ is ample, $H^0(X,\rho)$ is surjective for each divisible enough $m>0$. This means that $$f_*\O_X (m(K_{X/Y}+\e_2\Delta)-f^*H) \cong (f_*\O_X(m(K_{X/Y}+\e_2\Delta)))(-H)$$ is globally generated, and we have $t(Y,f_*\O_X(m(K_{X/Y}+\e_2\Delta)),H)\ge1>0$. \\
\tb{Step4.} Let $a$ be an integer not divisible by $p$ such that $a\Delta$ is integral. We show that $f_*\O_X(am(K_{X/Y}+\Delta))$ is ample for each $m\ge2/a$. 
Set $\e_3:=1/(p^b+1)$ for some integer $b\gg0$. Then, since $(X_{\ol\eta},(1+\e_3)\Delta_{\ol\eta})$ is $F$-pure by Fedder's criterion \cite{Fed83}, the proof of Theorem \ref{thm:main thm} shows that $$t(Y,f_*\O_X(m(K_{X/Y}+(1+\e_3)\Delta)),H)\ge0$$ for each divisible enough $m>0$. On the other hand, let $n_0\ge 2$ be an integer such that $$H^1(X_{\ol\eta},cn((1-\e_2-\e_3)K_{X_{\ol\eta}}+(1-\e_2-\e_3-2\e_2\e_3)\Delta_{\ol\eta}))=0$$ for some divisible enough integer $c>0$ and each $n\ge n_0$. This means that $$\O_{X_{\ol\eta}}(c(1-\e_2)n(K_{X_{\ol\eta}}+(1+\e_3)\Delta_{\ol\eta}))$$ is 0-regular with respect to $D_n:=c\e_3n(K_{X_{\ol\eta}}+\e_2\Delta_{\ol\eta}))$, where note that $D_n$ is a very ample divisor on $X_{\ol\eta}$. Hence the morphism \begin{align*}H^0(X_{\ol\eta},c(1-\e_2)n(K_{X_{\ol\eta}}+(1+\e_3)\Delta_{\ol\eta}))&\otimes H^0(X_{\ol\eta},D_n)\\ & \to H^0(X_{\ol\eta},c(1-\e_2+\e_3)n(K_{X_{\ol\eta}}+\Delta_{\ol\eta}))\end{align*} is surjective by the Castelnuovo-Mumford regularity \cite[Theorem~1.8.5]{Laz04}. 
From this, the morphism \begin{align*} f_*\O_X(c(1-\e_2)n(K_{X/Y}+(1+\e_3)\Delta)) &\otimes f_*\O_X(c\e_3n(K_{X/Y}+\e_2\Delta)) \\&\to f_*\O_X(c(1-\e_2+\e_3)n(K_{X/Y}+\Delta)) \end{align*} is generically surjective. By Lemma \ref{lem:t} and Step3, we have $$t(Y,f_*\O_X(c(1-\e_2+\e_3)n(K_{X/Y}+\Delta)),H)>0.$$ 
From now on we use notation of the proof of Theorem \ref{thm:main thm}. 
There exists an integer $l>0$ divisible enough such that $t(l):=t(Y,f_*\O_X(l(K_{X/Y}+\Delta)),H)>0$.
By an argument similar to Step2 in the proof of Theorem \ref{thm:main thm}, there exists a $h_0>0$ such that $t(l)+t(h)\le t(l+h)$ for every $h\ge h_0$. 
By Corollary \ref{cor:rs for F-pure curve}, we have $$S^0(X_{\ol\eta},\Delta_{\ol\eta},am(K_{X_{\ol\eta}}+\Delta_{\ol\eta}))=H^0(X_{\ol\eta},am(K_{X_{\ol\eta}}+\Delta_{\ol\eta}))$$ for each $m\ge 2/a$. 
Thus by an argument similar to Step3 and 4 in the proof of Theorem \ref{thm:main thm}, we have 
$$q_{m,e}t(l)+t(r_{m,e}+h_0)\le t(mp^e-a^{-1}(p^e-1))\le p^et(m),$$
for every $e>0$ with $a|(p^e-1)$.  Here, $q_{m,e}>0$ and $r_{m,e}$ respectively the quotient and the remainder of the devision of $mp^e-a^{-1}(p^e-1)-h_0$ by $l$.
Hence we obtain $t(Y,f_*\O_X(am(K_{X/Y}+\Delta)),H)>0$ for each $m\ge 2/a$, which implies $f_*\O_X(am(K_{X/Y}+\Delta))$ is ample by Proposition \ref{prop:t on curve}. This completes the proof.
\end{proof}
The following example shows that we can not drop the assumption $\lfloor \Delta_{\ol\eta}\rfloor=0$ in Theorem \ref{thm:ample}. 
\begin{eg}\label{eg:section} In the situation of Notation $\ref{notation:s-st}$, let $\Delta$ be an effective $\Q$-divisor on $X$, and let $C$ be a section of $f$. Taking blow-up, and replacing $\Delta$ and $C$ by their proper transforms, we may assume that $\Delta$ and $C$ are disjoint. Set $\Delta':=\Delta+C$. Then, for some integer $a>0$ such that $a\Delta$ is integral and for each $m\ge2/a$, there exists an exact sequence $$ 0\to \O_X(am(K_{X/Y}+\Delta')-C)\to \O_X(am(K_{X/Y}+\Delta'))\to\O_C\to0.$$ This induces a nonzero morphism $$f_*\O_X(am(K_{X/Y}+\Delta'))\to \O_Y,$$ since $$H^1(X_{\ol\eta},am(K_{X_{\ol\eta}}+\Delta'_{\ol\eta})-C_{\ol\eta})\cong H^0(X_{\ol\eta},(1-am)K_{X_{\ol\eta}}-am\Delta'_{\ol\eta}+C_{\ol\eta})=0.$$ Hence $f_*\O_X(am(K_{X/Y}+\Delta+C))$ is not ample for each $m\ge2$. \end{eg}
Next, in the situation of Notation $\ref{notation:s-st}$ we consider the positivity of $f_*\o_{X/Y}$. This relate to the $p$-ranks of fibers. Here the $p$-rank of a smooth projective curve $C$ is defined to be the integer $\gamma_C\ge0$ such that $p^{\gamma_C}$ is equal to the number of $p$-torsion points of the Jacobian variety of $C$. It is known that $\gamma_C=\dim_{k}S^0(C,\o_C)$. Jang proved that if the geometric generic fiber $X_{\ol\eta}$ is ordinary, (i.e., the $p$-rank of $X_{\ol\eta}$ is equal to $\dim_{k}H^0(X_{\ol\eta},\o_{X_{\ol\eta}})$), then $f_*\o_{X/Y}$ is a nef vector bundle \cite{Jan08}. Raynaud implied that if every closed fiber is a smooth ordinary curve, then $f$ is isotrivial \cite[TH\' EOR\` EME~5]{Szp81}. On the other hand, Moret-Bailly constructed an example of semi-stable fibration such that $f_*\o_{X/Y}$ is not nef \cite{MB81}. In this case, $X_{\ol\eta}$ is a smooth curve of genus two and of $p$-rank zero. As a generalization, Jang showed that if $f$ is non-isotrivial and the $p$-rank of $X_{\ol \eta}$ is zero, then $f_*\o_{X/Y}$ is not nef \cite{Jan10}. We generalize these results to the case of intermediate $p$-rank (Theorem \ref{thm:criterion}), based on the method in \cite{Jan10}. %
\begin{thm}[\textup{\cite[Corollary~2.5]{Jan10}}] \label{thm:fund ex seq} \samepage In the situation of Notation $\ref{notation:s-st}$, let $\mc M$ and $\mc T$ be the free part and the torsion part of $R^1{f_{Y^1}}_*\B^1_{X/Y}$ respectively, where $\B_{X/Y}^1$ is the kernel of $\phi^{(1)}_{X/Y}\otimes\o_{X_{Y^1}/{Y^1}}$ $($$\phi^{(1)}_{X/Y}$ is defined in Section $\ref{section:prelim}$$)$. Then there exists an exact sequence of $\O_{Y^1}$-modules $$ 0 \to \mc M^* \to {f^{(1)}}_*\o_{X^1/Y^1} \xrightarrow{{f_{Y^1}}_*(\phi^{(1)}_{X/Y}\otimes\o_{X_{Y^1}/{Y^1}})} F_Y^*f_*\o_{X/Y} \to \mc M\oplus \T \to 0. $$ \end{thm}
\begin{thm}[\textup{\cite[Proposition 2]{Szp79}}] \label{thm:Szp criterion}\samepage In the situation of Notation $\ref{notation:s-st}$, $f$ is isotrivial if and only if $\deg f_*\o_{X/Y}=0$. \end{thm} 
\begin{thm}[\textup{\cite[1.4. Satz]{LS77}}]\label{thm:fact on vb on sm curve}\samepage Let $\E$ be a vector bundle on a smooth projective curve $C$. If ${F_C^e}^*\E\cong\E$ for some $e>0$, then there exists an \'etale morphism $\pi:C'\to C$ from a smooth projective curve $C'$ such that $\pi^*\E\cong\bigoplus\O_{C'}$.  \end{thm}
\begin{thm}\label{thm:criterion} In the situation of Notation $\ref{notation:s-st}$, the function $$s(y):=\dim_{k(y)}S^0(X_y,\o_{X_y})$$ on $Y$ is lower semicontinuous. Furthermore, $f$ is isotrivial if and only if $s(y)$ is constant on $Y$ and $f_*\o_{X/Y}$ is nef.  \end{thm}
\begin{proof} 
By Observation \ref{obs:lower semi conti} (I\hspace{-1pt}V) we get the first statement (set $\Delta=0$ and $R=K_{X/Y}$).
We show that the second statement. If $f$ is isotrivial, then obviously $s(y)$ is constant, and by Example \ref{eg:isotrivial}, $f_*\o_{X/Y}$ is nef. Conversely, we assume that $s(y)$ is constant on $Y$. This means that $\coker(\theta^{(e)})$ is locally free for each $e\gg0$, where $$\theta^{(e)}:={f_{Y^e}}_*(\phi^{(e)}_{X/Y}\otimes\o_{X_{Y^{e}}/{Y^{e}}}):f^{(e)}_*\o_{X^e/Y^e}\to {f_{Y^e}}_*\o_{X_{Y^e}/{Y^e}}\cong {F_Y^e}^*f_*\o_{X/Y}.$$ Then there exists a commutative diagram $$\xymatrix{ f_*\o_{X/Y} \ar[r]^{\theta^{(e)}} \ar@{>>}[dr] \ar@/^30pt/[rr]^{\theta^{(e+1)}} & {F_Y^e}^*f_*\o_{X/Y} \ar[r]^{{F_Y^e}^*\theta^{(1)}} & {F_Y^{e+1}}^*f_*\o_{X/Y} \\ & \im(\theta^{(e)}) \ar@{^{(}-_{>}}[u] \ar@{>>}[dr]_{\alpha^{(e)}} & F_Y^*\im(\theta^{(e)}) \ar@{^{(}-_{>}}[u]\\ & & \im(\theta^{(e+1)}) \ar@{^{(}-_{>}}[u]_{\beta^{(e)}} } $$ for each $e>0$. When $e\gg0$, since $\alpha^{(e)}$ is a surjective morphism between vector bundles of the same rank, it is an isomorphism. Furthermore, since $\beta^{(e)}$ is an inclusion between subbundles of ${F_Y^{e+1}}^*f_*\o_{X/Y}$ of the same rank, it is an isomorphism. Thus we have $F_Y^*\im(\theta^{(e)})\cong\im(\theta^{(e)})$, in particular, by Theorem \ref{thm:fact on vb on sm curve}, $\deg(\im(\theta^{(e)}))=0$. Suppose that $f_*\o_{X/Y}$ is nef. Then, Theorem \ref{thm:fund ex seq} shows that $(\ker(\theta^{(1)}))^*$ is nef, hence that $(\ker(\alpha^{(e)}))^*$ is nef. 
Thus for every $e>0$, the exact sequence $$0\to \ker(\alpha^{(e)})\to \im(\theta^{(e)})\to \im(\theta^{(e+1)})\to 0$$ induces that 
$$\deg(\im(\theta^{(e)}))=\deg(\im(\theta^{(e+1)}))+\deg(\ker(\alpha^{(e)}))\le\deg(\im(\theta^{(e+1)})).$$ 
From this we have $\deg(\im(\theta^{(1)}))\le\deg(\im(\theta^{(2)}))\le\cdots\le0$, and hence 
$$0\le\deg f_*\o_{X/Y}=\deg(\ker(\theta^{(1)}))+\deg(\im(\theta^{(1)}))\le0.$$ 
Consequently, Theorem \ref{thm:Szp criterion} shows that $f$ is isotrivial, which completes the proof. \end{proof}
\section{Iitaka's conjecture}\label{section:iitaka}
In this section, we consider Iitaka's conjecture under the following hypotheses: 
\begin{notation}\label{notation:good fibrat}\samepage Let $f:X\to Y$ be a fibration between smooth projective varieties, let $\Delta$ be an effective $\Q$-divisor on $X$ such that $a\Delta$ is integral for some integer $a>0$ not divisible by $p$, and let $\ol\eta$ be the geometric generic point of $Y$. Assume that \begin{itemize} \item[(i)] $K_{X_{\ol\eta}}+\Delta_{\ol\eta}$ is finitely generated in the sense of Definition \ref{defn:fg}, and \item[(ii)] there exists an integer $m_0>0$ such that for every integer $m\ge m_0$ $$S^0(X_{\ol\eta},\Delta_{\ol \eta},m(aK_{X_{\ol \eta}}+(a\Delta)_{\ol \eta}))=H^0(X_{\ol \eta},m(aK_{X_{\ol \eta}}+(a\Delta)_{\ol \eta})).$$ \end{itemize} \end{notation}
Here condition (i) and (ii) are the same as in Theorem \ref{thm:main thm}. We first prove the case where $Y$ is of general type based on the method in \cite{Pat13}.
\begin{thm}\label{thm:iitaka conj Y gen}\samepage In the situation of Notation $\ref{notation:good fibrat}$, assume that $Y$ is of general type. Then $$ \kappa(X,K_X+\Delta)\ge\kappa(Y,K_Y)+\kappa(X_{\ol\eta},K_{X_{\ol\eta}}+\Delta_{\ol\eta}).$$ \end{thm}
In the proof we use the argument similar to \cite[\S 4]{Pat13} and the proof of \cite[Theorem~1.7]{Pat13}. 
\begin{proof} 
We may assume that $\kappa(X_{\ol\eta},K_{\ol\eta}+\Delta_{\ol\eta})\ge0$.\\
\tb{Step1.} 
Set $S':=\{\e\in\Q|\kappa(X,K_{X/Y}+\Delta-\e f^*H)\ge\kappa(X_{\ol\eta},K_{X_{\ol\eta}}+\Delta_{\ol\eta})+\kappa(Y)\}$, where $H$ is an ample divisor on $Y$. We show that $S'$ is nonempty. 
By the assumption (i), there exists an integer $b>0$ such that $R(X_{\ol\eta},ab(K_{X_{\ol\eta}}+\Delta_{\ol\eta}))$ is generated by $H^0(X_{\ol\eta},ab(K_{X_{\ol\eta}}+\Delta_{\ol\eta}))$. By projection formula, there exists an integer $c>0$ such that $f_*\O_X(ab(K_{X/Y}+\Delta)+cf^*H)$ is globally generated. Thus for every $m>0$, the natural morphism $$\bigotimes^{m}f_*\O_X(ab(K_{X/Y}+\Delta)+cf^*H)\to f_*\O_X(abm(K_{X/Y}+\Delta)+cmf^*H)$$ is generically surjective and shows that $f_*\O_X(abm(K_{X/Y}+\Delta)+cmf^*H)$ is generically globally generated.  This implies 
\begin{align*}\dim_k & H^0(X,abm(K_{X}+\Delta)+cmf^*H) \\
&\ge\dim_{k(\ol\eta)}H^0(X_{\ol\eta},abm(K_{X_{\ol\eta}}+\Delta_{\ol\eta}))+\dim_k H^0(Y,abmK_Y),
\end{align*}
Hence for $\e_0:=-c/(ab)$, we have $\kappa(X,K_X+\Delta-\e_0 f^*H)\ge\kappa(X_{\ol\eta},K_{X_{\ol\eta}}+\Delta_{\ol\eta})+\kappa(Y)$. \\
\tb{Step2.} Set $S:=\{\e\in\Q|\kappa(X,K_{X/Y}+\Delta-\e f^*H)\ge0\}$. We show that $\sup S=\sup S'$.
Since $S\supseteq S'$ we have the inequality $\ge$. We show $\le$. For an $\e\in S$, $K_{X/Y}+\Delta-\e f^*H$ is $\Q$-linearly equivalent to an effective $\Q$-divisor. Thus for every $0<\delta\in\Q$ and $\e_0\in S'$, 
\begin{align*}
\kappa(X,(1+\delta)(K_{X/Y}+\Delta)-(\e+\delta \e_0)f^*H)
\ge&\kappa(X,\delta(K_{X/Y}+\Delta-\e_0 f^*H)) \\
\ge&\kappa(X_{\ol\eta},K_{X_{\ol\eta}}+\Delta_{\ol\eta})+\kappa(Y).
\end{align*}
This implies $(\e+\delta \e_0)/(1+\delta)\le \sup S'$. Since $\lim_{\delta\to0}(\e+\delta \e_0)/(1+\delta)=\e$, we have $\e\le\sup S'$, and hence $\sup S\le\sup S'$. \\
\tb{Step3.} We show that $\sup S\ge0$.
For simplicity of notation, we denote $f_*\O_X(am(K_{X/Y}+\Delta))$ by $\G_m$. By the proof of \ref{thm:main thm}, we have $t(Y,\G_m,H)\ge0$ for each $m\ge m_0$. We fix an $m\ge m_0$ such that $\G_m\ne 0$, where such $m$ exists by the assumption that $\kappa(X_{\ol\eta},\Delta_{\ol\eta},K_{X_{\ol\eta}}+\Delta_{\ol\eta})\ge0$. 
Let $d>0$ be an integer such that $a|(p^d-1)$. Then for every $\e\in T(Y,\G_m,H)$ there exists an $e>0$ such that $p^{de}e\in\Z$ and $({F_Y^{de}}^*\G_m)(-p^{de}\e H)$ has a nonzero global section. 
On the other hand, since $X_{Y^{de}}$ is a variety, the natural morphism ${F^{(de)}_{X/Y}}^{\#}:\O_{X_{Y^{de}}}\to {F^{(de)}_{X/Y}}_*\O_{X^{de}}$ is injective, which induces injective $\O_{Y^{de}}$-module homomorphism 
\begin{align*}{F_Y^{de}}^*\G_m\cong {f_{Y^{de}}}_*\O_{X_{Y^{de}}}(am(K_{X_{Y^{de}}/Y^{de}}+\Delta_{Y^{de}})) \hookrightarrow {f^{(de)}}_*\O_{X^{de}}(amp^{de}(K_{X^{de}/Y^{de}}+\Delta)). \end{align*} 
Note that the reducedness of $X_{Y^{de}}$ follows from the separability of $f$ and the flatness of $F_{Y}$. 
From this $$H^0(X,amp^{de}(K_{X/Y}+\Delta)-p^{de}\e f^*H)\ne0,$$ and hence we have $(amp^{de})^{-1}p^{de}\e=(am)^{-1}\e\le \sup S$, and so $$0\le \frac{t(Y,\G_m,H)}{am}\le \sup S.$$ 
\tb{Step4.} We show the assertion. By the assumption and Step3, there exists an $\e\in S'$ such that $K_Y-\e H$ is linearly equivalent to an effective $\Q$-divisor. Then 
\begin{align*}
\kappa(X,K_X+\Delta)=&\kappa(X,K_{X/Y}+\Delta+f^*K_Y) \\
\ge&\kappa(X,K_{X/Y}+\Delta+\e f^*H)\ge\kappa(X_{\ol\eta},K_{X_{\ol\eta}}+\Delta_{\ol\eta})+\kappa(Y).
\end{align*}
This is our claim.
\end{proof}
Next, we show that Iitaka's conjecture when $Y$ is an elliptic curve (Theorem \ref{thm:iitaka conj Y ell}). To this end, we recall some facts about vector bundles on elliptic curves. 
\begin{thm}[\textup{\cite{Ati57,Oda71}}]\label{thm:facts on vb on ell curve}\samepage Let $C$ be an elliptic curve, and let $\E_C(r,d)$ be the set of isomorphism classes of indecomposable vector bundles of rank $r$ and of degree $d$. Then the following conditions are satisfied: \begin{itemize}\item[(1)] For each $r>0$, there exists a unique element $\E_{r,0}$ of $\E_C(r,0)$ such that $H^0(C,\E_{r,0})\ne0$. Moreover, for every $\E\in\E_C(r,0)$ there exists an $\L\in\Pic^0(C)=\E_C(1,0)$ such that $\E\cong\E_{r,0}\otimes\L$. \item[(2)] For every $\E\in\E_C(r,d)$, $$\left(\dim H^0(C,\E),\dim H^1(C,\E)\right)=\begin{cases} (d,0) & \textup{when $d>0$} \\ (0,-d) & \textup{when $d<0$} \\ (0,0) & \textup{when $d=0$ and $\E\ne\E_{r,0}$} \\ (1,1) & \textup{when $\E\cong\E_{r,0}$}. \end{cases} $$ \item[(3)] Let $\E\in\E_C(r,d)$. If $d>r$ $($resp. $d>2r$$)$ then $\E$ is globally generated $($resp. ample$)$. \item[(4)]\textup{(\cite[Corollary~2.9]{Oda71})} When the Hasse invariant ${\rm Hasse}(C)$ is nonzero, $F_C^*\E_{r,0}\cong \E_{r,0}$. When ${\rm Hasse}(C)=0$, $F_C^*\E_{r,0}\cong \bigoplus_{1\le i\le\min\{r,p\}}\E_{\lfloor(r-i)/p\rfloor+1,0}$.  \end{itemize} \end{thm}
\begin{rem}\label{rem:poten triv} Let $C$ be an elliptic curve, and let $r_0>0$ be an integer. When ${\rm Hasse}(C)=0$, Theorem \ref{thm:facts on vb on ell curve} (4) shows that there exists an $e>0$ such that ${F_C^e}^*\E_{r,0}\cong\bigoplus\O_C$ for each $r=1,\ldots,r_0$. When ${\rm Hasse}(C)\ne0$, Theorem \ref{thm:facts on vb on ell curve} (4) and Theorem \ref{thm:fact on vb on sm curve} show that there exists an \'etale morphism $\pi:C'\to C$ from an elliptic curve $C'$ such that $\pi^*\E_{r,0}\cong\bigoplus\O_{C'}$ for each $r=1,\ldots,r_0$. \end{rem}
We also need the following theorem. 
\begin{thm}[\textup{\cite[Theorem~10.5]{Iit82}}] \label{thm:cov thm}\samepage Let $f:X\to Y$ be a surjective morphism between smooth complete varieties, let $D$ be a divisor on $Y$, and let $E$ be an effective divisor on $X$ such that $\codim(f(E))\ge2$. Then $\kappa(X,f^*D+E)=\kappa(Y,D)$. \end{thm}
\begin{thm}\label{thm:iitaka conj Y ell}\samepage In the situation of Notation $\ref{notation:good fibrat}$, assume that $Y$ is an elliptic curve. Then $$ \kappa(X,K_X+\Delta)\ge\kappa(Y)+\kappa(X_{\ol\eta},K_{X_{\ol\eta}}+\Delta_{\ol\eta}).$$ \end{thm}
\begin{proof} \tb{Step1.} Let $M\ge m_0$ be an integer. We show that there exists a finite morphism $\alpha:Y'\to Y$ from an elliptic curve $Y'$ such that for each $m=m_0,m_0+1,\ldots,M$ $$\alpha^*f_*\O_X(am(K_{X/Y}+\Delta))\cong\F_m\oplus(\bigoplus_{\L\in S_m}\L),$$ where $\F_m$ is an ample and globally generated vector bundle, and $S_m\subseteq\Pic^0(Y')$. This claim follows from Theorem \ref{thm:main thm} and the lemma below. \begin{lem}\label{lem:nef vb on ell}\samepage Let $\E$ be a nef vector bundle on an elliptic curve $C$. Then there exists a finite morphism $\pi:C'\to C$ from an elliptic curve $C'$ such that $\pi^*\E\cong\F\oplus(\bigoplus_{\L\in S}\L)$, where $\F$ is an ample and globally generated vector bundle on $C'$, and $S\subseteq\Pic^0(C')$. \end{lem} \begin{proof}[Proof of Lemma $\ref{lem:nef vb on ell}$] By Theorem \ref{thm:facts on vb on ell curve} (3), replacing $\E$ by ${F_Y^e}^*\E$, we may assume that $\E\cong\F\oplus\G$ where $\F$ is ample and globally generated, and $\G$ is a direct sum of elements of $\bigcup_{r>0}\E_C(r,0)$. Hence, Theorem \ref{thm:facts on vb on ell curve} (1) and Remark \ref{rem:poten triv} complete the proof. \end{proof}
\noindent \tb{Step2.} We show that $\kappa(X_{Y'},K_{X_{Y'}}+\Delta_{Y'})=\kappa(X,K_X+\Delta)$. Obviously, we need only consider when $\a$ is separable and when $\a$ is purely inseparable. If $\a:Y'\to Y$ is separable then it is \'etale, thus so is $\a_X:X_{Y'}\to X$, in particular $X_{Y'}$ is a smooth variety. Hence the claim follows from Theorem \ref{thm:cov thm}, where we note that $K_{X_{Y'}}\sim K_{X_{Y'}/Y'}\sim (K_{X/Y})_{Y'}\sim (K_X)_{Y'}$.  If $\a=F_{Y/k}^{(e)}$ for some $e>0$, then there is a commutative diagram $$\xymatrix{ X^{e} \ar[d]_{F_{X/Y}^{(e)}} \ar[dr]^{F_{X/k}^{(e)}} && \\ X_{Y^e} \ar[r]_{(F_{Y/k}^{(e)})_X} \ar[d]_{f_{Y^e}} & X_{k^e} \ar[r]_{\cong} \ar[d]^{f_{k^e}} & X \ar[d]^{f} \\ Y^e \ar[r]_{F_{Y/k}^{(e)}} & Y_{k^e} \ar[r]_{\cong} & Y.}$$ Since $X_{Y^e}$ is a variety (cf. \cite[Lemma~5.2]{Pat13}), we have injective morphisms $\O_{X_{k^e}}\to {(F_{Y/k}^{(e)})_X}_*\O_{X_{Y^e}}\to {F_{X/k}^{(e)}}_*\O_{X^e}$, which induce injective morphisms $$ H^0(X,am(K_X+\Delta))\hookrightarrow H^0(X_{Y^e},am(K_{X_{Y^e}}+\Delta_{Y^e}))\hookrightarrow H^0(X^e,amp^e(K_{X^e}+\Delta))$$ for every $m>0$. Thus $\kappa(X,K_X+\Delta)=\kappa(X_{Y^e},K_{X_{Y^e}}+\Delta_{Y^e})$ as claimed. \\
\tb{Step3.} We complete the proof. Write $\G_m:=f_*\O_X(am(K_X+\Delta))$. Let $l,n_0>m_0$ be as in the proof of Theorem \ref{thm:main thm}. By the above argument, we may assume that $\G_m\cong\F_m\oplus(\bigoplus_{\L\in S_m}\L)$ for each $m\in\{l\}\cup\{n_0+i\}_{1\le i<l}$, where $\F_m$ is an ample and globally generated vector bundle, and $S_m\subseteq\Pic^0(Y)$. \\
\tb{Claim.} The subgroup $G$ of $\Pic^0(Y)$ generated by $S_l$ is a finite group. 
\begin{proof}[Proof of the claim.] Let $d,q_{l,e},r_{l,e}$ be as in the proof of Theorem \ref{thm:main thm} for each $e\gg0$. Set $S_l=\{\L_1,\ldots,\L_h\}$. Then for each $i=1,\ldots,h$, there exist generically surjective morphisms \begin{align*}  (\F_l\oplus\L_1\oplus\cdots\oplus\L_h)^{\otimes q_{l,e}}\otimes(\F_{n_0+r_{l,e}}\oplus ({\textstyle \bigoplus_{\L\in S_{n_0+r_{l,e}}}\L})) \cong  \G_l^{\otimes q_{l,e}}\otimes\G_{n_0+r_{l,e}}& \\ \to \G_{lp^{de}+a^{-1}(1-p^{de})} \to {F_Y^{de}}^*\G_l & \to \L_i^{p^{de}}\end{align*} as in the proof of Theorem \ref{thm:main thm}. It follows that there exists a nonzero morphism $\L_1^{t_1}\otimes\cdots\otimes\L_h^{t_h}\otimes\L\to \L_i^{p^{de}}$ for some integers $t_1,\ldots,t_h\ge0$ satisfying $\sum_{i=1}^h t_i=q_{l,e}$ and for some $\L\in\bigcup_{r=0}^{l-1}S_{n_0+r}$. Since this is a nonzero morphism between line bundles of degree zero on a smooth projective curve, this is an isomorphism, in particular $\L\in G$. 
For each $i=1,\ldots,h$ we denote $\L_i^{-1}$ by $\L_{i+h}$ for each $i=1,\ldots,h$, and for each $m>0$ we set $$\textstyle G(m):=\{\bigotimes_{i=1}^{2h}\L_i^{m_i}|\textup{ $0\le m_i$ and $\sum_{i=1}^{2h}m_i\le m$ }\}\subseteq G.$$ Let $c>0$ be an integer satisfying $\{\L\in G|\textup{$\L$ or $\L^{-1}$ is in $\bigcup_{r=0}^{l-1}S_{n_0+r}$}\}\subseteq G(c)$. Then by the above argument $\L_1^{p^{de}},\ldots,\L_{2h}^{p^{de}}\in G(q_{l,e}+c)$. Since $p^{de}>q_{l,e}+c$ for some $e\gg0$, there exists an $N>0$ such that $G=G(N)$, which is our claim. \end{proof}
By the claim, there exists an $n>0$ such that $n_Y^*\L\cong\L^n\cong\O_Y$ for each $\L\in S_l$. Hence, replacing $f$ by its base change with respect to $n_Y$, we may assume that $\G_l$ is globally generated. Then, for each $b\gg0$, $\G_{bl}$ is generically globally generated, because the natural morphism $\G_l^{\otimes b}\to \G_{bl}$ is generically surjective as in the proof of Theorem \ref{thm:main thm}. Thus we have \begin{align*}\dim_{k}H^0(X,abl(K_{X/Y}+\Delta))=& \dim_{k}H^0(Y,\G_{bl}) \\ \ge &\dim_{k(\ol\eta)}(\G_{bl})_{\ol\eta}=\dim_{k(\ol\eta)}H^0(X_{\ol\eta},bl(aK_{X_{\ol\eta}}+(a\Delta)_{\ol\eta}))\end{align*} for each $b\gg0$, and so $\kappa(X,K_X+\Delta)\ge\kappa(X_{\ol\eta},K_{X_{\ol\eta}}+\Delta_{\ol\eta})$.  
\end{proof}
There are some recent progress on Iitaka's conjecture in positive characteristic. Let $f:X\to Y$ be a fibration between smooth projective varieties, and let $X_{\ol\eta}$ be the geometric generic fiber. Chen and Zhang proved that $$\kappa(X)\ge\kappa(Y)+\kappa(Z)$$ when $f$ is of relative dimension one, where $Z$ is the normalization of $X_{\ol\eta}$ \cite[Theorem~1.2]{CZ13}. They also proved that $$\kappa(X)\ge\kappa(Y)+\kappa(X_{\ol\eta},K_{X_{\ol\eta}})$$ when $\dim X=2$ and $\dim Y=1$ \cite[Theorem~1.3]{CZ13}. Patakfalvi showed that $$\kappa(X)\ge\kappa(Y)+\kappa(X_{\ol\eta},K_{X_{\ol\eta}})$$ when $Y$ is of general type and $S^0(X_{\ol\eta},\o_{X_{\ol\eta}})\ne0$ \cite{Pat13}. For a related result, see \cite[Corollary~4.6]{Pat14}. \par 
On the other hand, as a direct consequence of Theorem \ref{thm:iitaka conj Y gen}, Theorem \ref{thm:iitaka conj Y ell}, and Corollary \ref{cor:surf of gen type}, we obtain the following new result:
\begin{cor}\label{cor:iitaka conj 3 1} \samepage Let $f:X\to Y$ be a fibration from a smooth projective variety $X$ of dimension three to a smooth projective curve $Y$. Assume that the geometric generic fiber $X_{\ol\eta}$ is a normal surface of general type with rational double point singularities, and $p\ge7$. Then $$\kappa(X)\ge\kappa(Y)+\kappa(X_{\ol\eta}).$$ \end{cor}
\begin{proof} We note that in this case $K_{X_{\ol\eta}}$ is finitely generated (cf. \cite[Corollary~9.10]{Bad01}). Thus the result follows from Corollary \ref{cor:surf of gen type} and Theorems \ref{thm:iitaka conj Y gen} and \ref{thm:iitaka conj Y ell}. \end{proof}

\end{document}